\documentclass[11pt]{amsart}

\usepackage{amscd,amssymb,amsopn,amsmath,amsthm,mathrsfs,graphics,amsfonts,enumerate,verbatim,calc,enumerate}
\usepackage[all]{xy}
\usepackage[shortlabels]{enumitem}
\usepackage{url}
\newlist{steps}{enumerate}{1}

\usepackage[OT2,OT1]{fontenc}
\newcommand\cyr{%
	\renewcommand\rmdefault{wncyr}%
	\renewcommand\sfdefault{wncyss}%
	\renewcommand\encodingdefault{OT2}%
	\normalfont
	\selectfont}
\DeclareTextFontCommand{\textcyr}{\cyr} 

\usepackage{amssymb,amsmath}

\DeclareFontFamily{OT1}{rsfs}{}
\DeclareFontShape{OT1}{rsfs}{n}{it}{<-> rsfs10}{}
\DeclareMathAlphabet{\mathscr}{OT1}{rsfs}{n}{it}

\numberwithin{equation}{section}
\newtheorem{theorem}{Theorem}[section]
\newtheorem{lemma}[theorem]{Lemma}
\newtheorem{proposition}[theorem]{Proposition}
\newtheorem{corollary}[theorem]{Corollary}

\newtheorem{question}{Question}

\newtheorem{maintheorem}{Main Theorem}

\theoremstyle{definition}
\newtheorem{definition}[theorem]{Definition}
\newtheorem{remark}[theorem]{Remark}

\theoremstyle{remark}

\newcommand{\Ass}{\operatorname{Ass}}
\newcommand{\im}{\operatorname{Im}}
\renewcommand{\ker}{\operatorname{Ker}}

\newcommand{\Spec}{\operatorname{Spec}}

\newcommand{\id}{\operatorname{id}}
\newcommand{\Ext}{\operatorname{Ext}}

\newcommand{\Supp}{\operatorname{Supp}}

\newcommand{\Hom}{\operatorname{Hom}}
\newcommand{\Att}{\operatorname{Att}}

\newcommand{\codim}{\operatorname{codim}}

\newcommand{\depth}{\operatorname{depth}}

\newcommand{\coker}{\operatorname{Coker}}

\newcommand{\Frac}{\operatorname{Frac}}

\newcommand{\Cl}{\operatorname{Cl}}

\newcommand{\Pic}{\operatorname{Pic}}
\newcommand{\Proj}{\operatorname{Proj}}

\newcommand{\alg}{\operatorname{alg}}

\newcommand{\fm}{\frak{m}}
\newcommand{\fp}{\frak{p}}
\newcommand{\fq}{\frak{q}}

\begin{document}
\title[Remarks on the SCM conjecture and new instances of MCM modules]
{Remarks on the Small Cohen-Macaulay conjecture and new instances of maximal Cohen-Macaulay modules}

\author[K.Shimomoto]{Kazuma Shimomoto}
\address{Department of Mathematics, College of Humanities and Sciences, Nihon University, Setagaya-ku, Tokyo 156-8550, Japan}
\email{shimomotokazuma@gmail.com}

\author[E.Tavanfar]{Ehsan Tavanfar}
\address{School of Mathematics, Institute for Research in Fundamental Sciences (IPM), P. O. Box: 19395-5746, Tehran, Iran}
\email{tavanfar@ipm.ir}

\thanks{
The research of the first author was partially supported by JSPS Grant-in-Aid for Scientific Research(C) 18K03257.\\
The research of the second author was supported by a grant from IPM}


\begin{abstract}
We show that any quasi-Gorenstein deformation of a $3$-dimensional quasi-Gorenstein Buchsbaum local ring with $I$-invariant $1$ admits a maximal Cohen-Macaulay module, provided it is a quotient of a Gorenstein ring. Such a class of rings includes two instances of unique factorization domains constructed by Marcel-Schenzel and by Imtiaz-Schenzel, respectively. Apart from this result, motivated by the small Cohen-Macaulay conjecture in prime characteristic, we examine a question about when the Frobenius pushforward $F^e_*(M)$ of an $R$-module $M$ comprises a maximal Cohen-Macaulay direct summand in both local and graded cases.
\end{abstract}

\maketitle
\tableofcontents

\section{Introduction}

The question on the existence of maximal Cohen-Macaulay modules over (Noetherian commutative) complete local  rings has been proposed in the seventies and it is known as \textit{the small Cohen-Macaulay conjecture}. It is an open question in Krull dimension $\ge 3$ in any characteristic. Here is a list of two wide classes of non-Cohen-Macaulay rings over which maximal Cohen-Macaulay modules have been found in the prime characteristic case:

\begin{enumerate}[(i)]
\item
\label{itm:FPureSMCholds}
Any $3$-dimensional complete local $\mathbb{F}_p$-algebra that is $F$-pure (\cite{SchoutensHochster}).

\item
\label{itm:GradedSCMholds}
Any $3$-dimensional complete local ring that arises as the completion of an $\mathbb{N}_0$-graded ring $R$ with $R_{[0]}$ a field of prime characteristic (\cite[Corollary 2]{HochsterCurrent}). 
\end{enumerate}

An essential part of the proof of \ref{itm:GradedSCMholds} is based on the following fact.

\begin{theorem}(\cite[Theorem 1]{HochsterCurrent} and its proof)
\label{Hochster}
Let $R$ be an $\mathbb{N}_0$-graded Noetherian domain of prime characteristic $p>0$.
\begin{enumerate}[(i)]
\item
\label{itm:VeroneseDecomposition}
For each $e\in \mathbb{N}$, the Frobenius direct image $F^e_*(R)$ of $R$  decomposes into $n_e$-number of non-zero direct summands $F^e_*(R)\cong \bigoplus_{i=1}^{n_e}R_{e,i}$ such that $\lim\limits_{e\rightarrow \infty}n_e=\infty$.
			
\item
\label{itm:SomeVeroneseDirectSummandIsCM}
If $R$ is generalized Cohen-Macaulay such that $R_{[0]}$ is a perfect field of characteristic $p>0$, then for each $e \gg 0$  there is some $1\le i\le n_e$ such that $R_{e,i}$ is a maximal Cohen-Macaulay $R$-module.
\end{enumerate}
\end{theorem}

One can observe that the proof of \cite[Theorem 1]{HochsterCurrent} implies the following local version.

\begin{theorem}
\label{FactNonGradedPrefectResidueField}
Let $R$ be an $F$-finite local ring with perfect residue field and let $M$ be a non-zero finitely generated $R$-module such that the following conditions hold.
\begin{enumerate}[(i)]
\item
\label{itm:FirstConditionOnM}
$M$ is generalized Cohen-Macaulay with $\dim(R)=\dim(M)$.
			
\item
\label{itm:SecondConditionOnM}
For each $e>0$, $F^e_*(M)$ admits a  decomposition $F^e_*(M) \cong \bigoplus_{i=1}^{n_e}M_{e,i}$ into $n_e$-number of non-zero  $R$-direct summands.
			
\item
\label{itm:ThirdConditionOnLimitInfity}
$\lim\limits_{e\rightarrow \infty}n_e=\infty$.
\end{enumerate}
Then for each $e\gg 0$, there is some $1\le i\le n_e$ such that $M_{e,i}$  is a maximal Cohen-Macaulay $R$-module.
\end{theorem}

In the nineties, Daniel Katz (\cite{KatzOnTheExistence}) proved the existence of maximal Cohen-Macaulay modules over the integral closure of some degree $p$ cyclic extensions of unramified regular local rings of mixed characteristic $p>0$. For other instances of non-Cohen-Macaulay rings over which  maximal Cohen-Macaulay modules have been discovered, we refer the reader to   \cite{MaMaximalCM}, \cite{SridharExistence}, \cite{SridharOnThe} and references therein (the main result of \cite{MaMaximalCM} shows that very small maximal Cohen-Macaulay modules do not exist in general).  Now let us state our first main result.

\begin{maintheorem}\label{MainTheorem1}(Theorem \ref{TheoremMaximalCohenMacaulayModule})
Let $(R,\fm,k)$ be a quasi-Gorenstein local ring which is a homomorphic image of a Gorenstein ring. Suppose that  there is a (possibly empty) regular sequence   $\mathbf{y}$   of $R$  such that $R/\mathbf{y}R$ is quasi-Gorenstein  and Buchsbaum of dimension $3$ and  $\ell\big(H^{2}_{\fm}(R/\mathbf{y}R)\big)=1$. Then $R$ admits a maximal Cohen-Macaulay module. Namely, there is some  $x\in 0:_RH^{\dim(R)-1}_{\fm}(R)$ such that $x,\mathbf{y}$ forms  a regular sequence of $R$. Let $\Omega_{R}$ be the first syzygy of the canonical module $\omega_{R/xR}$ of $R/xR$ in the minimal free resolution of $\omega_{R/xR}$ over $R$. Then $\Omega_{R}$ is a maximal Cohen-Macaulay $R$-module of rank $2$.
\end{maintheorem}

The following is a collection of results concerning Theorem \ref{Hochster} and Theorem \ref{FactNonGradedPrefectResidueField}, where the last one is a geometric manifestation of the small Cohen-Macaulay conjecture, the first one has no perfectness assumption on the residue field and  the residue field of  $R$ in (\ref{itm:Second}) is necessarily imperfect (more comments are given in the sequel).

\begin{maintheorem}
The following statement holds.
\begin{enumerate}\label{MainTheorem2}
\item(Proposition \ref{PropositionAlternativeProof1})
Let $(R,\fm)$ be an $F$-finite  $\mathbb{Z}$-graded local domain and let $M$ be a non-zero generalized Cohen-Macaulay $\mathbb{Z}$-graded $R$-module that is torsion-free (equivalently, $\fm\notin \Ass_R(M)$ and $\dim (M)=\dim (R)$). Then some Veronese module $M^{(p^e,i)}$ is a maximal Cohen-Macaulay $R$-module.

\item\label{itm:Second}(Proposition \ref{PropositionGradedTrueNonGradedWrong})
There is a torsion-free  generalized Cohen-Macaulay module $M$ over a local (non-graded) $F$-finite domain $(R,\fm)$ such that the following holds.
\begin{enumerate}
\item
$F^e_*(M)$ decomposes into $n_e$-number of non-zero $R$-direct summands $M_{e,0},\ldots,M_{e,n_e}$ with $\lim\limits_{e\rightarrow \infty}n_e=\infty $.

\item
No direct summand $M_{e,i}$ of $F_{*}^{e}(M)$ is a maximal Cohen-Macaulay module.
\end{enumerate}

\item(Corollary \ref{CorollaryProjectiveVariety})
Let $X$ be a normal projective variety of dimension $\ge 3$ over an $F$-finite field $k$ and let $D$ be an ample $\mathbb{Q}$-Weil divisor. Assume that $\mathcal{O}_X(\lfloor iD\rfloor)$ is an invertible sheaf for each $i\in \mathbb{Z}$, which is the case whenever $X$ is locally factorial, or $D$ is integral (and Cartier) or $R(X,D)$ satisfies the Goto-Watanabe condition $(\#)$.

\begin{enumerate} [(i)]
\item
\label{itm:MCMSheafEquivalentToExistenceOfACM}
$R(X,D)$ admits a graded maximal Cohen-Macaulay module if and only if there is a  maximal Cohen-Macaulay sheaf on $X$.

\item
\label{itm:LocallyFactorialIsNeedForEquivalenceSomeandAll}
If $X$ is locally factorial, then \label{itm:ACMisIndependeOfDivisor}  $R(X,D)$ admits a maximal Cohen-Macaulay module provided the  generalized section ring of any ample $\mathbb{Q}$-divisor over $X$ admits a maximal Cohen-Macaulay module.
\end{enumerate}
\end{enumerate}
\end{maintheorem}

\subsection{Contents of the article}

In \S \ref{SectionBackground}, we provide background material required  for Section \ref{SectionCMDirectSummandsOFFrobeniusDirectImages}.

 In \S   \ref{SectionquasiGorensteinDeformationsofCertainQuasiGorensteinBuchsbaums}   we present the proof of Main Theorem \ref{MainTheorem1}, and also in Remark \ref{RemarkExamples} we discuss certain instance of rings satisfying the assumptions of Main Theorem \ref{MainTheorem1}.

In \S \ref{SectionCMDirectSummandsOFFrobeniusDirectImages}, we are guided by  the following question:

\begin{question}
\label{QuestionStudiedInthePaper}
Assume that $R$ is a graded ring as in Theorem \ref{Hochster}\ref{itm:SomeVeroneseDirectSummandIsCM} that is not necessarily generalized Cohen-Macaulay. Then is it true that $F^e_*(R)$ admits a maximal Cohen-Macaulay $R$-direct summand for some $e>0$?
\end{question}

One motivation comes from \cite{TavanfarSmall}, where it is shown that the small Cohen-Macaulay conjecture holds true for an arbitrary complete local ring $R$, provided that certain blow-up algebra(s) of an excellent regular local ring $A$ such that $A$ surjects onto $R$ admit(s) a graded maximal Cohen-Macaulay module, so that the general case of the conjecture can be rearranged in a graded setting. In view of Corollary \ref{CorollaryTrivialDeformation}, there is an affirmative answer to Question \ref{QuestionStudiedInthePaper} at least in the case where $R$ is a trivial deformation of such a graded generalized Cohen-Macaulay ring. See also Proposition \ref{PropositionAlternativeProof1} and Corollary \ref{CorollaryVeroneseCMImpliesGeneralizedCohenMacaulay}. In relation with Question \ref{QuestionStudiedInthePaper}, we were curious about the necessity of perfectness assumption on the residue field of $R$ in Theorem \ref{Hochster}\ref{itm:SomeVeroneseDirectSummandIsCM} (graded case) and Theorem \ref{FactNonGradedPrefectResidueField} (local case), because after localizing a graded ring  $R$ as in Question \ref{QuestionStudiedInthePaper}, we get a generalized Cohen-Macaulay $F$-finite local ring, and while the Frobenius direct images behave well under localization, one loses perfectness of the residue field upon localization. In view of Proposition \ref{PropositionAlternativeProof1} and Proposition \ref{PropositionGradedTrueNonGradedWrong}, it becomes clear that the perfectness assumption is superfluous in the case of graded rings and Veronese submodules, but is necessary in the local setting.
	
The methods of birational geometry, alterations, or resolution of singularities can  be useful to produce maximal Cohen-Macaulay sheaves over a non-Cohen-Macaulay projective variety. In view of  Corollary \ref{CorollaryProjectiveVariety} and Remark \ref{RemarkProjToCone}\ref{itm:Demazure} under some reasonable conditions, we prove that an $\mathbb{N}_0$-graded normal domain $R$ over an $F$-finite field of prime characteristic admits a maximal Cohen-Macaulay module, provided that $\Proj (R)$ admits a maximal Cohen-Macaulay sheaf (see also Remark \ref{RemarkTavanfarReduction} and Remark \ref{RemarkAltration}). Lemma \ref{LemmaFPureGCMLocal} and Corollary \ref{CorollaryFPureGraded} provide further prime characteristic rings where small Cohen-Macaulay conjecture holds. These results have not been documented in the literature previously.

Summing up, the authors consider Section \ref{SectionCMDirectSummandsOFFrobeniusDirectImages} as an addendum to \cite[Theorem 1]{HochsterCurrent} as well as \cite{SchoutensHochster}. Almost all of the results of Section \ref{SectionCMDirectSummandsOFFrobeniusDirectImages} are interrelated, and more illustrative comments are to be found prior to each result.  We close the introduction by remarking that in the light of Schoutens' paper \cite{SchoutensADifferential}, there is a nice sufficient condition for the decomposability  (resp. direct summands) of the Frobenius direct images  in terms of existence of $F$-invariant derivations in the Kodaira-Spencer kernel (resp. $F$-invariant integrable derivations).

\section{Some preliminary background}
\label{SectionBackground}

Throughout this paper, rings are commutative with identity and are Noetherian. This section provides some background material for Section \ref{SectionCMDirectSummandsOFFrobeniusDirectImages}. Although the content of Section \ref{SectionquasiGorensteinDeformationsofCertainQuasiGorensteinBuchsbaums} is characteristic-free, but in Section \ref{SectionCMDirectSummandsOFFrobeniusDirectImages} rings are often of prime characteristic $p>0$ unless otherwise is mentioned explicitly, and are either local (i.e. quasi-local and Noetherian) or graded local. By a \textit{graded local ring} we mean a Noetherian $\mathbb{Z}$-graded ring $R$ satisfying the following condition.

\begin{enumerate}
\item[(Gr)]
$R$ admits a unique homogeneous maximal ideal that is maximal with respect to all  \textit{proper} ideals of $R$.
\end{enumerate}

A Noetherian $\mathbb{Z}$-graded ring $R$ is graded local precisely when $(R_0,\fm_0)$ is a local ring and
$$
\fm:=\cdots\oplus R_{[-1]}\oplus \fm_{0}\oplus R_{[1]}\oplus \cdots
$$
is the unique homogeneous maximal ideal of $R$. The unique homogeneous maximal ideal of $R$ satisfying $\mbox{(Gr)}$ is denoted by $R_{\pm}$, or $R_{+}$ in the positively graded case (if it is not denoted alternatively). The homogeneous localization of a graded module $M$ at a homogeneous prime ideal $\fp$ is denoted by $M_{(\fp)}$.\footnote{In contrast to our paper, this notation in some texts denotes the degree-zero submodule of the homogeneous localization.} An $R$-module $M$ acquires another $R$-module structure induced by restriction of the $e$-th iterated Frobenius map $f^e:R \to R$ defined by $r \mapsto r^{p^e}$. The resulting $R$-module structure is given by $r (^em)={^e(r^{p^e}m)}$, is called the \textit{($e$-th) Frobenius direct image of $M$} and is denoted by $F^e_*(M)$. If $R$ and $M$ are $\mathbb{Z}$-graded, then $F^e_*(M)$ is $\frac{1}{p^e}\mathbb{Z}$-graded as an $R$-module with $F^e_*(M)_{[\frac{n}{p^e}]}=M_{[n]}$ for each $n\in\mathbb{Z}$. We adopt the convention that all $\mathbb{Z}$-graded rings are non-trivially graded, that is to say, they admit non-zero homogeneous elements of positive degree. In Section \ref{SectionCMDirectSummandsOFFrobeniusDirectImages}, we often work with $F$-finite rings. The $F$-finiteness condition is necessary to ensure that all Frobenius direct images of finitely generated modules and their direct summands are finitely generated. Moreover, a theorem of Kunz says that $F$-finite rings are excellent.  In this paper, a projective variety is assumed to be an integral closed subscheme of some projective space $\mathbb{P}^n$ and in particular, they are all connected. The perfect closure of a field $K$ is denoted by $K^\infty$.

\subsection{Graded direct summands of the Frobenius direct images}

\begin{definition}
\label{FactorizationRemark}
Let $R$ be a Noetherian $\mathbb{Z}$-graded ring of prime characteristic $p>0$ and let $M$ be a finitely generated $\mathbb{Z}$-graded $R$-module. For each $e \in\mathbb{N}$ and $0\le i\le p^e-1$, the \textit{Veronese submodule} of $M$ is  defined by
$$
M^{(p^e,i)}:=\bigoplus_{\substack{j\in\mathbb{Z}\\j{\equiv}i \pmod{p^e}}}M_{[j]}=\cdots \oplus M_{[i-2p^e]} \oplus M_{[i-p^e]} \oplus M_{[i]} \oplus M_{[i+p^e]} \oplus \cdots.
$$
\end{definition}

An important ingredient of the proof  \cite[Theorem 1]{HochsterCurrent} is the fact that each  $M^{(p^e,i)}$ acquires an $R$-module structure by which it is a (graded) $R$-direct summand of $F^e_*(M)$. On the other hand, $M^{(p^e,i)}$ is a module over the $p^e$-th Veronese subring $R^{(p^e)}=\bigoplus_{j \in \mathbb{Z}}R_{[jp^e]}$ of $R$, which is Noetherian. It seems that this $R^{(p^e)}$-module structure of  $M^{(p^e,i)}$ is not pointed out in \cite[Theorem 1]{HochsterCurrent}. However, we are sure that this is well-known to experts.

\begin{remark}
\label{RemarkDifferentPerspective}
Suppose that $R$ is an $F$-finite $\mathbb{Z}$-graded local domain of characteristic $p>0$ and $M$ is a finitely generated $\mathbb{Z}$-graded $R$-module.

\begin{enumerate}[(i)]
\item
\label{itm:VeroneseSubmoduleOverVeronseSubring}
$M^{(p^e,i)}$ has a graded $R^{(p^e)}$-module structure given by the naive action of $R$ on $M$.  Furthermore, one considers $R$ and $M$ as $R^{(p^e)}$-modules via the inclusion $R^{(p^e)}\hookrightarrow R$, so that
$$
R \cong \bigoplus_{i=0}^{p^e-1}R^{(p^e,i)},~M \cong \bigoplus_{i=0}^{p^e-1}M^{(p^e,i)}~\mbox{as}~R^{(p^e)}\text{-graded modules}.
$$

\item
\label{itm:VeroneseSubmoduleOverOriginalRing}
The image of the $e$-th iterated Frobenius map $f^e:R\rightarrow R$ is contained in $R^{(p^e)}$. In particular, $f^e$ makes $R^{(p^e)}$ an $R$-algebra, and $M^{(p^e,i)}$ can be endowed with an $R$-module structure by restriction of its $R^{(p^e)}$-module structure via $f^{p^e}:R\rightarrow R^{(p^e)}$.

\item
\label{itm:VeroneseSubmoduleDirectSummandOfFrobeniusDirectImage}
Consider each $M^{(p^e,i)}$ with its $R$-module structure as in the previous part. Then the $R$-module $F_*^e(M)$ admits a decomposition into direct summands
$$
F_*^e(M) \cong \bigoplus_{i=0}^{p^e-1}M^{(p^e,i)}~\mbox{as}~R\mbox{-graded modules}.
$$

\item
\label{itm:VeroneseSubmoduleCM}
$M^{(p^e,i)}$ is a maximal Cohen-Macaulay $R$-module if and only if it is maximal Cohen-Macaulay as a $R^{(p^e)}$-module via the Independence Theorem \cite[4.2.1]{BrodmannSharpLocal}.

\item
\label{itm:VeronseSubringMCMImplies}
$R$ admits a maximal Cohen-Macaulay module if $R^{(p^e)}$ admits an $R^{(p^e)}$-maximal Cohen-Macaulay module for some $e>0$. Indeed, $R$ is assumed to be $F$-finite and the iteration of the Frobenius map on $R$ factors through $R^{(p^e)}$, so $R^{(p^e)}$ is a module-finite extension of $R$. This happens when $R^{(p^e)}$ is a subring of a regular ring as in \cite[Example 3.6]{ShimomotoFCoherent}, or $R$ is generalized Cohen-Macaulay (see Proposition \ref{PropositionAlternativeProof1}). 

\item
\label{itm:FrobeniusDirectImageOfVeronseIsDirctSumOfVeronese}
It is readily seen that $\big(M^{(p^e,i)}\big)^{(p^{e'},j)}=M^{(p^{e+e',jp^e+i})}$, which yields the ecomposition
$$
F^{e'}_*(M^{(p^e,i)})\cong \bigoplus_{j=0}^{p^{e'}-1}M^{(p^{e+e'},jp^e+i)}~\mbox{as}~R\mbox{-graded modules}.
$$

\item
\label{itm:MTorsionFreeImpliesGoesToInfinity}
Assume that $R$ is an integral domain and $M$ is a torsion-free $R$-module. Then the number of non-zero Veronese submodules $M^{(p^e,i)} \subset F^e_*(M)$ tends to infinity when $e\rightarrow \infty$. Indeed, they carry infinitely many non-zero positively graded components. Pick some $n,m \in \mathbb{N}$ with $R_{[n]}\neq 0$ and $M_{[m]}\neq 0$. Then $0\neq M_{[nkp^e+ni+m]}\subseteq M^{(p^e,ni+m)}$ for each $k\in \mathbb{N}$ and $0\le i\le \lfloor (p^e-1-m)/n\rfloor$ and sufficiently large $e$.
\end{enumerate}
\end{remark}

\subsection{$\mathbb{Q}$-Weil divisors on normal projective varieties}

The content of this subsection is characteristic-free. There is a nice connection between $\mathbb{Z}$-graded modules over normal $\mathbb{N}_0$-graded rings whose degree zero part is a field and rational coefficient ample Weil divisors on normal projective varieties. This connection will be reviewed below, following our need in  Section \ref{SectionCMDirectSummandsOFFrobeniusDirectImages}.

\begin{definition}
\label{DefinitionGotoWatanbeCondition}
Following \cite[page 206]{GotoWatanabeOnGraded}, an $\mathbb{N}_0$-graded ring satisfies the \textit{Goto-Watanabe condition} $(\#)$, provided there is some positive integer $d_0$ such that for each $d\ge d_0$ the Veronese subring, $R^{(d)}$ of $R$ is generated by $R_{[d]}={R^{(d)}}_{[1]}$ over $R_{[0]}$.
\end{definition}

\begin{remark}
\label{RemarkSectionRingOfAmpleGotoWatanabeCondition}
The class of graded rings satisfying $(\#)$, which contains standard graded rings, is nice for algebro-geometric objects. As stated in \cite[in the last line at page $207$ without a proof]{GotoWatanabeOnGraded},  the section ring of an ample integral Cartier divisor  $D$ on a normal projective variety satisfies the Goto-Watanabe condition $(\#)$. Indeed by \cite[Theorem  1.2.6 (Cartan-Serre-Grothendieck Theorem)]{LazarsfeldPositivity}, $dD$ is very ample for $d\gg0$, which implies that $R(X,dD)$ is a standard graded ring for $d\gg 0$.
\end{remark}

We recall our assumption that a projective variety is an integral closed subscheme of $\mathbb{P}^n$ for some $n\in \mathbb{N}$, thus it is a connected scheme. For the notion of an ample $\mathbb{Q}$-divisor on a normal projective variety, we refer to \cite[page 204]{WatanabeSomeRemarks}.

\begin{definition}
\label{DefinitionGeneralizedSectionRing}
Let $D$ be an ample $\mathbb{Q}$-divisor on a normal projective variety $X$. Then the \textit{generalized section ring} of $D$ is defined as
$$
R(X,D):=\bigoplus_{{i\in \mathbb{N}_0}} H^0\big(X,\mathcal{O}_X(\lfloor iD\rfloor)\big).
$$
Let $\mathcal{F}$ be a coherent sheaf on $X$. Then the \textit{graded module associated to the pair $(\mathcal{F},D)$} is defined to be the $\mathbb{Z}$-graded $R(X,D)$-module
$$
\Gamma_*(X,\mathcal{F},D):=\bigoplus_{i\in \mathbb{Z}} H^0\big(X,\mathcal{F}\otimes\mathcal{O}_X(\lfloor iD\rfloor)\big).
$$
\end{definition}

\begin{remark}
\label{RemarkProjToCone}
Following the notation of Definition \ref{DefinitionGeneralizedSectionRing}, let $N \in \mathbb{N}$ be such that $ND$ is an ample Cartier divisor.
\begin{enumerate}[(i)]
\item
\label{itm:Demazure}
Recall the following result of Demazure \cite[Theorem, page 51]{DemazureAnneaux}:

\begin{enumerate}
\item[$\bullet$]
If $R$ is a positively graded Noetherian normal domain with $R_{[0]}=K$ a field and $T$ is a degree $1$ homogeneous element of $\Frac(R)$, then there exists an ample $\mathbb{Q}$-divisor $D$ on $X:=\Proj(R)$ such that
$$
R \cong \bigoplus_{i\ge 0}H^0\big(X,\mathcal{O}_X(\lfloor iD\rfloor)\big){T}^{i}.
$$
\end{enumerate}

For an arbitrary $\mathbb{N}_0$-graded Noetherian normal domain $R$ over a field $R_{[0]}=K$, without loss of generality, we can assume that the greatest common divisor of the generators of $R$ over $K$ is $1$. Indeed, we can change the grading by dividing out by their greatest common divisor and the defining ideal of $R$ remains homogeneous with respect to the given new grading. We can deduce that $\Frac(R)$ admits a homogeneous element of degree $1$ and so we can apply the Demazure's theorem to present $R$ as a generalized section ring of an ample $\mathbb{Q}$-divisor.
			
\item
\label{itm:tiwstedassociatedSheaf}
Setting $M:=\Gamma_*(X,\mathcal{F},D)$, for each $i\in \mathbb{Z}$ we have
$$
\widetilde{M(i)}=\mathcal{F}\otimes \mathcal{O}_X(\lfloor iD\rfloor). 
$$

\item
\label{itm:LocalCohomolgyVSSheafCohomolgy}
For each $m\in \mathbb{Z}, i\ge 1$, we have
$$
H^{i+1}_{R(X,D)_+}\big(\Gamma_*(X,\mathcal{F},D)\big)_{[m]}=H^{i}\big(X,\mathcal{F}\otimes \mathcal{O}_X(\lfloor mD \rfloor)\big).
$$
This follows from part \ref{itm:tiwstedassociatedSheaf} in conjunction with \cite[Proposition 2.1.5]{GrothendieckEGAIII}.
			
\item
\label{itm:ExactSequenceAssociatedGradedModuleAndOriginalModule} (\cite[(5.1.6)]{GotoWatanabeOnGraded})
If $D$ is  integral (and Cartier) and $M$ is a $\mathbb{Z}$-graded $R(X,D)$-module such that $\widetilde{M}=\mathcal{F}$, then there is an exact sequence
$$
0 \rightarrow \Gamma_{R(X,D)_+}(M)\rightarrow M\rightarrow \Gamma_*(X,\mathcal{F},D) \rightarrow H^1_{R(X,D)_+}(M)\rightarrow 0.
$$

\item
\label{itm:Grothendieck}
If $\Ass_X\big(\mathcal{F}\otimes \mathcal{O}_X(\lfloor iD\rfloor)\big)$ does not contain any closed point of $X$ for each $0\le i\le N-1$, then $\Gamma_*(X,\mathcal{F},D)$ is a finitely generated  graded $R(X,D)$-module.
			
\item
\label{itm:DepthOfAssociatedGradedModule}
If $\Ass_X\big(\mathcal{F}\otimes \mathcal{O}_X(\lfloor iD\rfloor)\big)$ does not contain any closed point of $X$ for each $0\le i\le N-1$, then $\text{depth}\big(\Gamma_*(X,\mathcal{F},D)\big)\ge 2$.
			
\item
\label{itm:FactVeroneseModulesFiniteOverVeronesSubring}
$R(X,D)$ is module-finite over the Veronese subring $R^{(m)}$ (see the proof of \cite[Proposition 1.1.2.5]{CoxRing}). In particular, every $R^{(m,j)}$ is a finitely generated $R^{(m)}$-module.
\end{enumerate}
\end{remark}

\begin{remark}
\label{RemarkSemiAmple} In general, in view of \cite[Example 2.1.29]{LazarsfeldPositivity} and \cite[Proposition 1.1.2.5]{CoxRing}, the section ring of any $\mathbb{Q}$-Weil divisor $D$ such that  $ND$ is an integral semi-ample Cartier divisor for some $N\in \mathbb{N}$ is a Noetherian ring. However, because of the discernment of normal graded domains in terms of generalized section rings of ample $\mathbb{Q}$-divisors (Remark \ref{RemarkProjToCone}\ref{itm:Demazure}), we confine ourselves only to ample $\mathbb{Q}$-divisors, as the small Cohen-Macaulay conjecture is easily reduced to the normal case.
\end{remark}

\section{Quasi-Gorenstein deformations of quasi-Gorenstein Buchsbaum rings of dimension $3$ and $I$-invariant $1$} \label{SectionquasiGorensteinDeformationsofCertainQuasiGorensteinBuchsbaums}

A \textit{quasi-Buchsbaum ring} is a local ring $(R,\fm)$ whose non-top local cohomologies supported at $\fm$ are annihilated by $\fm$. The Buchsbaum rings appearing in our paper are quasi-Buchsbaum (\cite[Proposition 2.12]{StuckradVogelBuchsbaum}), and we do not use the notion and properties of $\fm$-weak sequences. So the reader of the paper can omit the definition and properties of Buchsbaum rings. However, for the definition and properties of Buchsbaum rings  we refer to \cite{StuckradVogelBuchsbaum}. Also, see  \cite[Definition 2.1]{MaMaximalCM} for the definition of the \textit{$I$-invariant}, and see \cite[Definition 2.1 and Definition 1.1]{AoyamaSomeBasic} for the definition of quasi-Gorenstein rings as well as the canonical module. It is well-known that any homomorphic image of a local Gorenstein ring admits a canonical module and  a dualizing complex. A Noetherian ring $R$ is said to be \textit{unmixed} provided $\dim(R/\fp)=\dim(R)$ for all $\fp \in \Ass(R)$. The notation $(-)^{\vee}$ denotes the Matlis duality functor and $\ell_R(M)$ denotes the length of an $R$-module $M$.

The main result of this section is Theorem \ref{TheoremMaximalCohenMacaulayModule}, which gives an instance of maximal Cohen-Macaulay modules over certain non-Cohen-Macaulay rings.  We first need a preliminary lemma.

\begin{lemma}
\label{LemmaExtOfCanonicalModule}
Let $(A,\fm)$ be a homomorphic image of some Gorenstein ring with a canonical module $\omega_A$.  
\begin{enumerate}[(i)]
\item
\label{itm:LemmaExtOfCanonicalModuleSpectralSequence} There is an exact sequence
$$
H^{\dim(A)-2}_\fm(\omega_A)\rightarrow \Ext^2_A(\omega_A,\omega_A)^\vee\rightarrow H^{\dim(A)-1}_\fm(A)\otimes \omega_A\rightarrow H^{\dim(A)-1}_\fm(\omega_A)\rightarrow \Ext^1_A(\omega_A,\omega_A)^\vee\rightarrow 0.
$$

\item
\label{itm:LemmaExtOfCanonicalModuleVanishing}
If $H^{\dim(A)-1}_{\fm}(\omega_A)=0$ (e.g. if $\omega_A$ is a maximal Cohen-Macaulay $A$-module), then
$$
\Ext^1_{A}(\omega_A,\omega_A)=0.
$$
\end{enumerate}
\end{lemma}

\begin{proof}
There exists a Gorenstein local ring $S$ such that $\dim(S)=\dim(A)$ and $A$ is a homomorphic image of $S$. The statement \ref{itm:LemmaExtOfCanonicalModuleSpectralSequence} follows from the Grothendieck spectral sequence
$$
\Ext^i_A\big(\omega_A,\Ext^j_S(A,S)\big)\overset{i,j}{\Rightarrow}\Ext^{i+j}_S(\omega_A,S)
$$
in conjunction with \cite[Theorem 10.33 (Cohomology Five-Term Exact Sequence)]{Rotman} as well as the Local Duality Theorem \cite[11.2.6]{BrodmannSharpLocal}. Part \ref{itm:LemmaExtOfCanonicalModuleVanishing} is an immediate consequence of part \ref{itm:LemmaExtOfCanonicalModuleSpectralSequence}.
\end{proof}

\begin{theorem}
\label{TheoremMaximalCohenMacaulayModule}
Let $(R,\fm,k)$ be a quasi-Gorenstein local ring which is a homomorphic image of a Gorenstein ring. Suppose that  there is a (possibly empty) regular sequence   $\mathbf{y}$   of $R$  such that $R/\mathbf{y}R$ is quasi-Gorenstein  and Buchsbaum of dimension $3$ and  $\ell\big(H^{2}_{\fm}(R/\mathbf{y}R)\big)=1$. Then $R$ admits a maximal Cohen-Macaulay module. Namely, there is some  $x\in 0:_RH^{\dim(R)-1}_{\fm}(R)$ such that $x,\mathbf{y}$ forms  a regular sequence of $R$. Let $\Omega_{R}$ be the first syzygy of the canonical module $\omega_{R/xR}$ of $R/xR$ in the minimal free resolution of $\omega_{R/xR}$ over $R$. Then $\Omega_{R}$ is a maximal Cohen-Macaulay $R$-module of rank $2$.
\end{theorem}

\begin{proof}
 We first aim to find an element $x$ as required in the statement. 

\textbf{Step 1 (existence of element $x$ as in the statement)}: Since $R$ and $R/\mathbf{y}R$ are quasi-Gorenstein by our hypothesis, they are unmixed rings, because $R$ and $R/\mathbf{y}R$ are $S_2$ by \cite[(1.10)]{AoyamaSomeBasic} and then the desired unmixedness follows from \cite[Lemma 1.1]{AoyamaGotOnTheEndomorphism}. This implies that $(\mathbf{y})$ is an unmixed ideal of $R$, i.e. $\codim(\mathbf{y}R)=\codim(\fp)$ for each $\fp\in \Ass(\mathbf{y}R)$. We show that $$0:_RH^{\dim(R)-1}_{\fm}(R)\nsubseteq \bigcup\limits_{\fp\in \Ass(\mathbf{y}R)}\fp,$$ from which we deduce the existence of an element $x \in 0:_RH^{\dim(R)-1}_{\fm}(R)$ such that $x,\mathbf{y}$  is a regular sequence. Suppose to the contrary that
$$
0:_RH^{\dim(R)-1}_{\fm}(R)\subseteq \fp
$$
for some $\fp\in \Ass(\mathbf{y}R)$. Then there exists $\fq\in \Att\big(H^{\dim(R)-1}_{\fm}(R)\big)$ such that $\fq\subseteq \fp$ (\cite[Proposition 7.2.11]{BrodmannSharpLocal}). Thus,
$$
\fq R_\fp\in \Att_{R_\fp}\big(H^{\dim(R)-\dim(R/\fp)-1}_{\fp R_{\fp}}(R_{\fp})\big)=\Att_{R_\fp}\big(H^{\dim(R_{\fp})-1}_{\fp R_{\fp}}(R_{\fp})\big),
$$
where the containment holds by \cite[Theorem 3.7]{NhanQuyAttached} 
and the equality holds by \cite[Lemma 2, page 250]{Matsumura}. But $R_\fp$ is a Gorenstein ring, because  $\codim(\mathbf{y}R)=\codim(\fp)$ and $\mathbf{y}$ is a regular sequence, and we get a contradiction with   $H^{\dim(R_{\fp})-1}_{\fp R_{\fp}}(R_{\fp})=0$.


For the rest of the proof, we may and do assume that $R$ is complete.

\textbf{Step 2 (reduction to the case of dimension $3$ as in the statement)}: We show that  the statement  reduces readily to the case where $R$ is a $3$-dimensional Buchsbaum ring as in the statement. Note that $\mathbf{y}$ forms a regular sequence on $R/xR$ by our choice of $x$. We aim to prove that   $\mathbf{y}$ forms a regular sequence on $\omega_{R/xR}$ and that $\omega_{R/xR}/\mathbf{y}\omega_{R/xR}\cong \omega_{R/(x,\mathbf{y})}$. Then the regularity of $\mathbf{y}$ on $\omega_{R/xR}$ implies that  $\mathbf{y}$ is a regular sequence on $\Omega_R$ as well. It follows from the isomorphism $\omega_{R/xR}/\mathbf{y}\omega_{R/xR}\cong \omega_{R/(x,\mathbf{y})}$ as well as \cite[Proposition 1.1.5]{BrunsHerzogCohenMacaulay} that we have $\Omega_{R/\mathbf{y}R}\cong \Omega_R/\mathbf{y}\Omega_R$, where $\Omega_{R/\mathbf{y}R}$ denotes the first syzygy of $\omega_{R/(x,\mathbf{y})}$ in the minimal free resolution of $\omega_{R/(x,\mathbf{y})}$ over $R/\mathbf{y}R$. It turns out that if the statement holds for the quasi-Gorenstein Buchsbaum ring $R/\mathbf{y}R$ as in the statement, then $\Omega_{R}$ is a maximal Cohen-Macaulay $R$-module. Hence the proof of this step is completed once we can show that  $\mathbf{y}$ forms  a regular sequence on $\omega_{R/xR}$ and   $\omega_{R/xR}/\mathbf{y}\omega_{R/xR}\cong \omega_{R/(x,\mathbf{y})}$.

Let $\mathbf{y}=y_1,\ldots,y_h$ and $d=\dim(R)$. We  notice that $\depth(R)=\dim(R)-1$. Namely, $R/\textbf{y}R$ is a 3-dimensional quasi-Gorenstein ring. On the other hand, since any quasi-Gorenstien local ring is $(S_2)$ and $H^2_{\fm}(R/\mathbf{y}R)\cong k\neq 0$ by our hypothesis, it follows that $R/\mathbf{y}R$ is not Cohen-Macaulay.

 
Using a standard argument based on the long exact sequence of local cohomologies, it is easily seen that
\begin{align}
\label{EquationLocalCohomologyInvaraintUnderQuotient}
H^{d-2}_{\fm}(R/xR)\cong H^{d-1}_{\fm}(R)~\mbox{for}~x \in 0:{H^{d-1}_{\fm}(R)}, 
\end{align}
that
\begin{equation}\label{EquationLocalCohomologOfQuotientIsAnnihilatorOfLocalCohomology}
\forall\  1\le i\le h,\ \  H^{d-i}_{\fm}\big(R/(y_1,\ldots,y_{i-1})\big)=0:_{H^{d-1}_{\fm}(R)}(y_1,\ldots,y_{i-1}),
\end{equation}
and similarly that
\begin{align}
\label{EquationLocalCohomologOfQuotientIsAnnihilatorOfLocalCohomology2}
\nonumber \forall\  1\le i\le h,\ \  H^{d-i-1}_{\fm}\big(R/(x,y_1,\ldots,y_{i-1})\big)&=0:_{H^{d-2}_{\fm}(R/xR)}(y_1,\ldots,y_{i-1})
&\\& =0:_{H^{d-1}_{\fm}(R)}(y_1,\ldots,y_{i-1}), && (\text{by (\ref{EquationLocalCohomologyInvaraintUnderQuotient})}).
\end{align}

Thus from our hypothesis, we get 
\begin{align}
\label{EquationQuasiGorensteinRole}
\ \forall\  1\le i\le h,\ \  y_i&\notin \bigcup\limits_{\fp\in \Att\Big(H^{d-i}_{\fm}\big(R/(y_1,\ldots,y_{i-1})\big)\Big)}\fp, && (\text{by \cite[Corollary 2.8]{TavanfarTousiAStudy}, \cite[Corollaire (5.12.4)]{GrothendieckEGAIV}})
\nonumber &\\&\overset{}{=}\bigcup\limits_{\fp\in \Att\big(0:_{H^{d-1}_{\fm}(R)}(y_1,\ldots,y_{i-1})\big)}\fp, && (\text{by (\ref{EquationLocalCohomologOfQuotientIsAnnihilatorOfLocalCohomology})})
			  \nonumber &\\&=\bigcup\limits_{\fp\in \Att\Big(H^{d-1-i}_{\fm}\big(R/(x,y_1,\ldots,y_{i-1})\big)\Big)}\fp, && (\text{by (\ref{EquationLocalCohomologOfQuotientIsAnnihilatorOfLocalCohomology2})}).\end{align}

Thus, a repeated use of \cite[Theorem 2.7]{TavanfarTousiAStudy} implies that $$\forall\ 1\le i\le h, \ \ \omega_{R/xR}/(y_1,\ldots,y_i)\omega_{R/xR}\cong \omega_{R/(x,y_1,\ldots,y_i)}$$ and for each $0\le i\le h-1$, the element $y_{i+1}$ is  regular on $\omega_{R/xR}/(y_1,\ldots,y_i)\omega_{R/xR}$ by \cite[(1.10)]{AoyamaSomeBasic}. So the proof of Step 2 is complete.

Henceforth, we assume that $R$ is a complete $3$-dimensional quasi-Gorenstein Buchsbaum ring of $I$-invariant $1$ (i.e. $\mathbf{y}$ is the empty sequence). Let $$S:=\Hom_{R/xR}(\omega_{R/xR},\omega_{R/xR})$$ be the $S_2$-ification of $R/xR$ which is a (possibly semi-local) Gorenstein ring by \cite[Theorem 2.6]{TavanfarTousiAStudy} (any $2$-dimensional quasi-Gorenstein ring is Gorenstein), and in view of \cite[Theorem 12.3.10(v) and Theorem 2.2.6(c)]{BrodmannSharpLocal}, it fits into the exact sequence of $R/xR$-modules
\begin{equation}        
\label{EquationExactSequenceS2ification}
0\rightarrow R/xR\overset{f}{\rightarrow} S\overset{\pi}{\rightarrow} H^1_{\fm}(R/xR)\big(\cong H^2_{\fm}(R)\cong k\big)\rightarrow 0.
\end{equation}
As $\omega_{S_\fp}\cong (\omega_{R/xR})_{\fp}$ for each $\fp\in \Spec(S)$ (\cite[Theorem 3.2(4)]{AoyamaSomeBasic}), from the Gorensteinness of $S$ as well as \cite[Lemma 1.4.4]{BrunsHerzogCohenMacaulay} we conclude that $\omega_{R/xR}\cong S$ as $R,R/xR$ and $S$-modules. In particular,  $\omega_{R/xR}$ is generated minimally by $2$ elements in view of (\ref{EquationExactSequenceS2ification}). Notice that $\omega_{R/xR}$ is a faithful $R/xR$-module, so it is not cyclic. Otherwise $\omega_{R/xR}$ would be a free $R/xR$-module, i.e. $R/xR$  and so $R$ would be Gorenstein violating the non-Cohen-Macaulayness of $R$ ($R$ is $(S_2)$, so $R/xR$ is unmixed and $\omega_{R/xR}$ is faithful by \cite[Theorem 12.1.15]{BrodmannSharpLocal}). 

\textbf{Step 3}: $\Ext^2_R(S,R)\overset{\text{\cite[Lemma 2, page 140]{Matsumura}}}{\cong} \Ext^1_{R/xR}(S,R/xR)=0:$

\vspace{2mm}
Applying $\Hom_{R/xR}(-,R/xR)$ to (\ref{EquationExactSequenceS2ification}) we get the embedding 
\begin{equation}
\label{EquationEmbedding}
\Hom_{R/xR}(S,R/xR)\overset{\Hom(f,\id_{R/xR})}{\hookrightarrow} \Hom_{R/xR}(R/xR,R/xR) \cong R/xR,
\end{equation}
which we claim that it is not surjective.

If otherwise, we get some isomorphism $S^{**}\overset{\cong}{\rightarrow} R/xR$, where $(-)^{**}$ denotes the double $R/xR$-dual. Composing this isomorphism with the natural biduality map $S\rightarrow S^{**}$, we get an $R/xR$-monomorphism $h:S\rightarrow R/xR$ such that $h$ is an isomorphism on the punctured spectrum of $R/xR$, because $R$ and $R/xR$ are Gorenstein on their punctured spectrums (the injectivity of $h$ follows from \cite[Proposition 1.4.1(a)]{BrunsHerzogCohenMacaulay}). Then, the multiplicative map $h\circ f:R/xR\rightarrow R/xR$ must be an isomorphism on the punctured spectrum of $R/xR$.  Since $R/xR$ has dimension $2$, $h\circ f$ has to be a multiplication by a unit and therefore an isomorphism. However, then the exact sequence (\ref{EquationExactSequenceS2ification}) would split and $k$ would be a direct summand of $S$ which is a contradiction. It follows that the map in (\ref{EquationEmbedding}) is not surjective as was claimed. Since $\Hom(f,\id_{R/xR})$ is not surjective as just shown, we have
\begin{equation}
\label{EquationMapIsZeroModuleKFirstVersion}
\im\big(\Hom(f,\id_{R/xR})\big)\subseteq \fm \Hom_{R/xR}(R/xR,R/xR).\end{equation}

Next by applying $\Hom_{R/xR}(S,-)$ to the exact sequence (\ref{EquationExactSequenceS2ification}), we can take into account the map $\Hom(\id_S,f):\Hom_{R/xR}(S,R/xR)\rightarrow \Hom_{R/xR}(S,S)$ for which we claim that
\begin{equation}
\label{EquationMapIsZeroModuloK}
\im\big(\Hom_{R/xR}(S,R/xR)\overset{\Hom(\id_S,f)}{\longrightarrow} \Hom_{R/xR}(S,S)\big)\subseteq \fm\Hom_{R/xR}(S,S).
\end{equation}

In order to certify the validity of $(\ref{EquationMapIsZeroModuloK})$, we take into account the triangle 
$$
\xymatrix{\Hom_{R/xR}(S,R/xR)\ar[rrr]^{\Hom(\id_S,f)} \ar[drrr]_{\Hom(f,\id_{R/xR})} &&& \Hom_{R/xR}(S,S)\\ &&& \Hom_{R/xR}(R/xR,R/xR) \ar[u]_\theta}
$$
where $\theta$ is defined by $r\id_{R/xR}\mapsto r\id_{S}$. In order to conclude (\ref{EquationMapIsZeroModuloK}) from (\ref{EquationMapIsZeroModuleKFirstVersion}) it suffices to show that the above triangle is commutative. We notice that $\Hom(f,\id_{R/xR})(\varphi)=\varphi \circ f$ while $\Hom(\id_{S},f)(\varphi)=f\circ \varphi$. Let $T$ be the multiplicative closed subset of $R/xR$ consisting of regular elements of $R/xR$, so $T^{-1}(R/xR)$ is the total ring of fractions of $R/xR$. Since  $T^{-1}(f)$ is an isomorphism it is readily checked that  the triangle becomes commutative after localizing at $T$ (by representing $T^{-1}(\varphi)$ ($\varphi\in \Hom_{R/xR}(S,R/xR)$) as $T^{-1}(\varphi)=\alpha\id_{T^{-1}(R/xR)}\circ \big(T^{-1}(f)\big)^{-1}$ for some $\alpha\in T^{-1}(R/xR)$). So the statement follows from the fact that $S$ has no torsion elements. It follows that (\ref{EquationMapIsZeroModuloK}) holds.

Tensoring the  exact sequence  $$0\rightarrow \Hom_{R/xR}(S,R/xR)\overset{\Hom(\id_S,f)}{\rightarrow} \Hom_{R/xR}(S,S)\rightarrow N:=\im\big(\Hom(\id_S,\pi)\big)\big(\subseteq \Hom_{R/xR}(S,k)\big)\rightarrow 0$$ (obtained from (\ref{EquationExactSequenceS2ification})) with $k$ yields the sequence 
$$
\Hom_{R/xR}(S,R/xR)/\fm \Hom_{R/xR}(S,R/xR)\overset{\nu=0}{\rightarrow} \Hom_{R/xR}(S,S)/\fm\Hom_{R/xR}(S,S)\overset{\zeta}{\twoheadrightarrow} N\overset{\delta}{\hookrightarrow} \Hom_{R/xR}(S,k),
$$  
where $\ker(\zeta)=\im(\nu)$, $\zeta$ is surjective, $\delta$ is injective because  the inclusion of $k$-vector spaces $N\subseteq \Hom_{R/xR}(S,k)$ remains injective (remains unchanged) after tensoring to $k$,  and  finally $\nu=0$ in view of (\ref{EquationMapIsZeroModuloK}). It follows that $\delta\circ \zeta$ is a monomorphism of $k$-vector spaces. But then the map $\delta\circ \zeta$ has to be an isomorphism, because its source and target are $2$-dimensional $k$-vector spaces, as $S\cong \omega_S\cong \omega_{R/xR}$ is minimally generated by $2$ elements over $R/xR$ and so is $\Hom_{R/xR}(S,S)\cong \Hom_{R/xR}(\omega_{R/xR},\omega_{R/xR})=S$. The surjectivity of $\delta\circ \zeta$ implies that
\begin{equation}
\label{EquationSurjectivity}
\Hom_{R/xR}(\id_S,\pi):\Hom_{R/xR}(S,S)\rightarrow \Hom_{R/xR}(S,k)
\end{equation}  
is surjective. From this surjectivity  in conjunction with the long exact sequence
$$
0\rightarrow \Hom_{R/xR}(S,R/xR)\rightarrow \Hom_{R/xR}(S,S) \rightarrow \Hom_{R/xR}(S,k)\rightarrow \Ext^1_{R/xR}(S,R/xR)\rightarrow \underset{0}{\underbrace{\Ext^1_{R/xR}(S,S)}}
$$
arising from (\ref{EquationExactSequenceS2ification}), we conclude that 
$\Ext^1_{R/xR}(S,R/xR)=0$. We recall that the vanishing of the last Ext module in the above exact sequence follows from $S\cong \omega_{R/xR}$ as well as Lemma \ref{LemmaExtOfCanonicalModule}\ref{itm:LemmaExtOfCanonicalModuleVanishing}.
			
\textbf{Step 4 (the final step)}: Let
$$
({D}^\bullet,\partial^{\bullet}):=0\rightarrow D^0\overset{\partial^0}{\rightarrow} D^1\overset{\partial^1}{\rightarrow} D^2\overset{\partial^3}{\rightarrow} D^3\rightarrow 0
$$
be the normalized dualizing complex of $R$ (see \cite[Definition 11.2.1]{SchenzelSimonCompletion} for the definition of the dualizing complex and  \cite[11.4.6 and Example 11.4.7]{SchenzelSimonCompletion} for the notion of the normalized dualizing complex), and let
$$
\tau^0{D}^\bullet:=0\rightarrow \im\partial^{0}\rightarrow D^1\rightarrow D^2\rightarrow D^3\rightarrow 0
$$
be the soft truncation of $D^\bullet$ which thus fits into an exact sequence $0\rightarrow R(\cong \omega_R)\rightarrow {D}^\bullet\rightarrow \tau^{0}{D}^\bullet\rightarrow 0$ of complexes of $R$-modules (by \cite[Corollary 12.2.3]{SchenzelSimonCompletion} $\omega_R\cong H^0(D^\bullet)=\ker(\partial^{0})$). This exact sequence  and  $0\rightarrow \Omega_R\rightarrow R^2\rightarrow S(\cong \omega_{R/xR})\rightarrow 0$ provide us with the diagram
\begin{center}
$\begin{CD}
@. 0 @. 0 @. 0\\
@.  @AAA  @AAA  @AAA \\
0 @>>> R\Hom_R(\Omega_R,R) @>>> R\Hom_R(\Omega_R,{D}^\bullet) @>>> R\Hom_R(\Omega_R,\tau^{0}{D}^\bullet)@>>> 0\\
@.  @AAA  @AAA  @AAA \\
0 @>>> R\Hom_R(R^2,R) @>>> R\Hom_R(R^2,{D}^\bullet) @>>> R\Hom_R(R^2,\tau^{0}{D}^\bullet)@>>> 0\\
@.  @AAA  @AAA  @AAA \\
0 @>>> R\Hom_R(S,R) @>>> R\Hom_R(S,{D}^\bullet) @>>> R\Hom_R(S,\tau^{0}{D}^\bullet)@>>> 0\\
@.  @AAA  @AAA  @AAA \\
@. 0 @. 0 @. 0
\end{CD}$
\end{center}
whose rows and columns are exact sequence of complexes. After taking cohomology, we get the diagram 
\begin{center}
$\begin{CD}
@. \underset{=0\text{\ (depth$(S)=2$)}}{\underbrace{H^1_{\fm}(S)^\vee=H^2\big(R\Hom_R(S,{D}^\bullet)\big)}}\\
@. @AAA\\
@. H^{2}_\fm(\Omega_R)^\vee=H^1\big(R\Hom_R(\Omega_R,{D}^\bullet)\big)\\
@. @AAA\\
\underset{=0}{\underbrace{\Ext^1_R(R^2,R)}} @>>> H^2_{\fm}(R)^\vee\oplus H^2_{\fm}(R)^\vee @>\zeta_2>\cong>   H^1\big(R\Hom(R^2,\tau^0{D}^\bullet)\big)@>>>\underset{=0}{\underbrace{\Ext_R^2(R^2,R)}}\\
@. @A\zeta_1AA @A\zeta_4AA\\
@. S\cong \omega_S\cong H^1\big(R\Hom_R(S,{D}^\bullet)\big) @>\zeta_3>>  H^1(R\Hom_R(S,\tau^0{D}^\bullet)) @>>> \underset{=0\text{\ (Step 3)}}{\underbrace{\Ext_R^2(S,R)}}
\end{CD}$
\end{center}
with exact rows and columns (the long column holds by Local Duality; see \cite[Theorem 12.2.1]{SchenzelSimonCompletion}). To finish the proof, it suffices to show that $\zeta_1$ is surjective, because then we get $H^2_{\fm}(\Omega_R)=0$ while we already know that $\depth(\Omega_R)\ge 2$. In view of the above diagram, the surjectivity of $\zeta_1$ follows from the surjectivity of $\zeta_4$. We will show that $\zeta_4$ is surjective and then the proof is complete.

By Local Duality, we have $\tau^0D^\bullet\simeq (H^{2}_{\fm}(R)^\vee)^{[-1]} =k^{[-1]}$ in the derived category of $R$, where $[-1]$ denotes the complex shift for ascending complexes as defined in \cite[1.1.1 Terminology about complexes]{SchenzelSimonCompletion}. Consequently,
\begin{align*}
\coker(\zeta_4)&=\coker\Big(H^1\big(R\Hom_R(S,\tau^0D^\bullet)\big)\rightarrow H^1\big(R\Hom_R(R^2,\tau^0D^\bullet)\big)\Big) 
\\&\cong\coker\Big(H^1\big(R\Hom_R(S, k^{[-1]})\big)\rightarrow H^1\big(R\Hom_R(R^2, k^{[-1]})\big)\Big) 
\\&\cong \coker\big(\Hom_R(S,k)\rightarrow \Hom(R^2,k)\big)
\end{align*}
which vanishes, because $R^2\twoheadrightarrow S$ becomes an isomorphism after tensoring to $k$ as $S$ is minimally generated by $2$ elements.
\end{proof}

We end this section by the following remarks to Theorem \ref{TheoremMaximalCohenMacaulayModule}.

\begin{remark}
We make comments on the importance of the conditions imposed in Theorem \ref{TheoremMaximalCohenMacaulayModule} for the Cohen-Macaulayness of  $\Omega_R$.
\begin{enumerate}[(i)] 
\item
The ``quasi-Gorenstein deformation" condition in Theorem \ref{TheoremMaximalCohenMacaulayModule} may not be replaced by an arbitrary deformation in general. Firstly, the quasi-Gorenstein property does not deform in general by \cite{ShimomotoTaniguchiTavanfar}. Therefore, an arbitrary deformation of a quasi-Gorenstein Buchsbaum ring as in Theorem \ref{TheoremMaximalCohenMacaulayModule} is not necessarily quasi-Gorenstein. Secondly, the quasi-Gorenstein property of the deformations in Theorem \ref{TheoremMaximalCohenMacaulayModule} plays an important role in proving the validity of the display (\ref{EquationQuasiGorensteinRole}) in the proof of   Step 2.

\item
The ``dimension $3$" condition for the quasi-Gorenstein Buchsbaum ring $R/\mathbf{y}R$ in Theorem \ref{TheoremMaximalCohenMacaulayModule} seems to be necessary, because without this condition neither the canonical module of $R/(\mathbf{y},x)R$ nor (equivalently)  its $S_2$-ification is Cohen-Macaulay.

\item
The ``$I$-invariant $1$" condition in Theorem \ref{TheoremMaximalCohenMacaulayModule} is also essential, because without this assumption, the conclusion of Step 3 in the proof does not hold in general. Namely there is a Buchsbaum quasi-Gorenstein ring of $I$-invariant $2$ and dimension $3$ such that $\Ext^2_R(S,R)\neq 0$, where $S$ is the $S_2$-ification of $R/xR$ for some (probably any) non-zero divisor $x$ of $R$. However, certain $3$-dimensional quasi-Gorenstein Buchsbaum rings with $I$-invariant $2$ do admit a maximal Cohen-Macaulay module; for instance, certain section rings of an Abelian surface over an algebraically closed field of any characteristic (see \cite{ShimomotoTavanfarOnLocalRings}).
\end{enumerate}
\end{remark}

\begin{remark}
\label{RemarkExamples}
We mention prime characteristic  instances of unique factorization domains satisfying the assumptions of Theorem \ref{TheoremMaximalCohenMacaulayModule}. At the time of writing this paper, the authors do not know any example of an equicharacteristic zero ring satisfying the assumptions of Theorem \ref{TheoremMaximalCohenMacaulayModule}, while  there are instances of equicharacteristic zero $3$-dimensional quasi-Gorenstein Buchsbaum rings with $I$-invariant $2$, for example a variant of  \cite[Example 5.1(2)]{ShimomotoTaniguchiTavanfar} or any generalized section ring of an Abelian surface over the complex numbers, or \cite[Example, page 624]{HerrmannTrungExamples}. Anyway, such an example, if it exists in equal characteristic zero, indeed can not occur as a $3$-dimensional complete unique factorization domain with algebraically closed residue field of characteristic zero (in fact of any $I$-invariant), becuase such a ring has to be Cohen-Macaulay by \cite[Theorem 2.5]{HartshorneOgusOnTheFactoriality}.

\begin{enumerate}[(i)]
\item
One instance of rings satisfying the assumptions of Theorem \ref{TheoremMaximalCohenMacaulayModule} is constructed by Marcelo and Schenzel in \cite[Theorem 2.5]{MarceloSchenzelNonCMUFDs}, which is a unique factorization domain of dimension $3$ in characteristic $2$. Apparently, the defining ideal of the ring $A$ in \cite[Theorem 2.5]{MarceloSchenzelNonCMUFDs} is not homogeneous and therefore (if it is so) the fact that this ring satisfies the small Cohen-Macaulay Conjecture does not follow from Theorem \ref{Hochster} or Proposition \ref{PropositionAlternativeProof1} (the defining ideal is explicitly given in the paragraph before \cite[Theorem 2.5]{MarceloSchenzelNonCMUFDs}).

\item
By Macaulay2 computations, we have checked that the Imtiaz-Schenzel construction in \cite[Theorem 6.4]{ImtiazSchenzel}, which is a graded ring whose defining ideal is explicitly given in its preceding paragraph, provides another instance of rings adhering the assumptions of Theorem \ref{TheoremMaximalCohenMacaulayModule}. The Imtiaz-Schenzel ring is a unique factorization domain of dimension $5$ and of characteristic $2$. Moreover, this ring   is not generalized Cohen-Macaulay and so the fact that it adheres the small Cohen-Macaulay conjecture does not follow from Theorem \ref{Hochster} (or Proposition \ref{PropositionAlternativeProof1}). It is asked as a question in \cite{ImtiazSchenzel} that whether the authors' construction gives such non-Cohen-Macaulay unique factorizations in other characteristics or not (see the paragraph after \cite[Theorem 1.2]{ImtiazSchenzel}). 

\item
While the constructions in the previous parts were done by factorial closure, invariant theory can also provide further examples. Namely, using Macaulay2 program, we checked that the famous Bertin's example of a non-Cohen-Macaulay unique factorization domain in characteristic $2$ (\cite[Example 16.8]{FossumTheDivisor}) is also a $4$-dimensional unique factorization domain that is an invariant subring of a regular ring satisfying the assumptions of Theorem \ref{TheoremMaximalCohenMacaulayModule}. The Bertin's example is not generalized Cohen-Macaulay.  However, since this is an invariant ring, it admits even a small Cohen-Macaulay algebra.
\end{enumerate}	
\end{remark}

\section{Construction of maximal Cohen-Macaulay modules via the Frobenius map} \label{SectionCMDirectSummandsOFFrobeniusDirectImages}

One trick used in the proof of Theorem \ref{Hochster}\ref{itm:SomeVeroneseDirectSummandIsCM} is the following. If the residue field is perfect, we have
$$
\ell_R(H^i_\fm(M))=\ell_R(F_*^e(H^i_\fm(M)))=\ell_R(H^i_\fm(F^e_*(M)).
$$
This is important, because the similar length computation is valid and useful for non-graded local rings with prefect residue field. On the other hand, by arguing with   maximum/minimum degree among non-zero components instead of length identities,  one  can relax the perfectness assumption in the hypothesis of Theorem \ref{Hochster}\ref{itm:SomeVeroneseDirectSummandIsCM}.

\begin{proposition}
\label{PropositionAlternativeProof1}
Let $(R,\fm)$ be an $F$-finite  $\mathbb{Z}$-graded local domain and let $M$ be a non-zero generalized Cohen-Macaulay $\mathbb{Z}$-graded $R$-module that is torsion-free (equivalently, $\fm\notin \Ass_R(M)$ and $\dim (M)=\dim (R)$). Then some Veronese module $M^{(p^e,i)}$ is a maximal Cohen-Macaulay $R$-module. 
\end{proposition}

\begin{proof}
Set $d:=\dim (R)$. If $M$ is torsion-free, then $\Ass_R(M)=\{0\}$ so that $\dim(M)=\dim(R)$ and $\fm\notin \Ass_R(M)$. Since $M$ is generalized Cohen-Macaulay and $R$ is an integral domain, it follows from $\fm\notin \Ass_R(M)$ and $\dim(M)=\dim(R)$ that $\Ass_R(M)=\{0\}$ (i.e. $M$ is torsion-free) by applying \cite[9.5.2 Grothendieck's Finiteness Theorem]{BrodmannSharpLocal}, which yields $\lambda^{\fm}_{\fm}(M)=\dim(R)$.  

  Let $\mathbf{x}:=x_1,\ldots,x_n$ be a set of homogeneous elements of $R$ generating an $\mathfrak{m}$-primary ideal. By our hypothesis, there is some $j \in \mathbb{N}$ with $M_{[j]}\neq 0$ (thus $M^{(p^e,j)}\neq 0$ for any $e$ with $p^e>j$) and

\begin{equation}
\label{EquationZeroDegrees}
H^i_{(\mathbf{x}^{p^e})}(M)_{[h]}=H^i\big(\mbox{C}^\bullet(\mathbf{x}^{p^e};M)\big)_{[h]}=0~\mbox{for}~|h|\ge j~\mbox{and}~i<d.
\end{equation}

Note also that $\dim(M^{(p^e,j)})=d$, as $M$ is torsion-free for all $e$ with $p^e>j$. Pick some $e$ with $p^e \ge 2j$; thus

\begin{equation} 
\label{EquationSimpleInequatiy}
\forall\ k \in \mathbb{Z},\ \ |kp^e+j|\ge j. 
\end{equation}

By Remark \ref{RemarkDifferentPerspective}\ref{itm:VeroneseSubmoduleOverVeronseSubring}, we have the homogeneous decomposition
$$
H^i_{(\mathbf{x}^{p^e})}(M)=\bigoplus_{u=0}^{p^e-1}H^i_{(\mathbf{x}^{p^e})}(M^{(p^e,u)})
$$
into $R^{(p^e)}$-graded direct summands.  In view of the \v{C}ech-complex presentation of the local cohomology as well as (\ref{EquationSimpleInequatiy}),   the graded direct summand $H^i_{(\mathbf{x}^{p^e})}(M^{(p^e,j)})$ of $H^i_{(\mathbf{x}^{p^e})}(M)$ is contained componentwise in $H^i_{(\mathbf{x}^{p^e})}(M)$ among degrees either greater than or equal to $j$, or smaller than or equal to $-j$. Therefore, (\ref{EquationZeroDegrees}) implies that $\bigoplus_{i=0}^{d-1}H^i_{(\mathbf{x}^{p^e})}(M^{(p^e,j)})=0$, that is, $M^{(p^e,j)}$ is a Cohen-Macaulay $R^{(p^e)}$-module. So the statements follows from Remark \ref{RemarkDifferentPerspective}\ref{itm:VeroneseSubmoduleCM}.
\end{proof}

\begin{remark}
The domain hypothesis in the statement of Proposition \ref{PropositionAlternativeProof1} is required for excluding the case allowing the graded ring to have a nilpotent irrelevant ideal.
\end{remark}

\subsection{Cohen-Macaulay sheaves on normal projective varieties}

\begin{definition}
A coherent sheaf $\mathcal{F}$ on a scheme $X$ is a said to be \textit{maximal Cohen-Macaulay}, provided $\mathcal{F}_x$ is a maximal Cohen-Macaulay $\mathcal{O}_{X,x}$-module for any $x \in X$.
\end{definition}

The following lemma is required for Corollary \ref{CorollaryProjectiveVariety} and Corollary \ref{CorollaryVeroneseCMImpliesGeneralizedCohenMacaulay}.   It is noteworthy that Lemma \ref{LemmaMCMSheafToModule} is characteristic-free.

\begin{lemma}
\label{LemmaMCMSheafToModule}
Let $X$ be a normal projective variety of dimension $\ge 2$ over a field $k$ of any characteristic and let $\mathcal{F}$ be a maximal Cohen-Macaulay sheaf on $X$. Let $D$ be an ample $\mathbb{Q}$-Weil divisor on $X$. If $\mathcal{O}_{X}(\lfloor iD\rfloor)$ is an invertible sheaf for each $i\in \mathbb{Z}$, then $\Gamma_*(X,\mathcal{F},D)$ is a generalized Cohen-Macaulay $\mathbb{Z}$-graded $R(X,D)$-module. In particular, $\Gamma_*(X,\mathcal{F},D)$ is generalized Cohen-Macaulay  if either of the following conditions holds:
\begin{enumerate}[(i)]
\item
\label{itm:LocallyFactorial}
$X$ is locally factorial.

\item
\label{itm:GotoWatanabeCondition}
$R(X,D)$ satisfies the  Goto-Watanabe condition $(\#)$ (e.g., if $R(X,D)$ is standard graded).

\item
\label{itm:IntegralAmpleDivisor}
$D$ is an (integral) ample Cartier divisor.
\end{enumerate}
\end{lemma}

\begin{proof}
Note that $M=\Gamma_*(X,\mathcal{F},D)$ is a finitely generated module over $R(X,D)$ by Remark \ref{RemarkProjToCone}\ref{itm:Grothendieck}. In the case   \ref{itm:LocallyFactorial}, every Weil divisor is Cartier, so $\mathcal{O}_{X}(\lfloor iD\rfloor)$ is an invertible sheaf for every $i\in \mathbb{Z}$ (\cite[Proposition 6.11 or Corollary 6.16]{HartshorneAlgebraicGeometry}). In the case \ref{itm:IntegralAmpleDivisor}, we are in the situation of \ref{itm:GotoWatanabeCondition} by Remark \ref{RemarkSectionRingOfAmpleGotoWatanabeCondition}. Finally, in the case where \ref{itm:GotoWatanabeCondition} holds, all $\mathcal{O}_X(\lfloor iD\rfloor)$ are invertible sheaves by \cite[Lemma 5.1.2]{GotoWatanabeOnGraded}. Thus, it suffices to prove the lemma by assuming that all rank $1$ reflexive sheaves $\mathcal{O}_{X}(\lfloor iD \rfloor)$ are invertible sheaves. We divide the proof into two steps.

\underline{Step $1$: $D$ is a very ample integral Cartier divisor.}
In this case, there is a closed immersion $f:X \to \mathbb{P}^n$ such that $\mathcal{O}_X(D) \cong f^*\big(\mathcal{O}_{\mathbb{P}^n}(1)\big)$. Setting $\mathcal{G}:=f_*(\mathcal{F})$, we have
$$
H^i\big(X,\mathcal{F}\otimes\mathcal{O}_X(mD)\big)=H^i\big(\mathbb{P}^n,\mathcal{G}\otimes \mathcal{O}_{\mathbb{P}^n}(m)\big)
$$
for each $m$ and $i$ by \cite[Corollary 3]{KempfSomeElementaryProofs} as well as the Projection Formula (see \cite[Proposition (12.2.3)]{GrothendieckEGAIII} and also \cite[II, Exercise 5.1(d)(Projection Formula)]{HartshorneAlgebraicGeometry}). Note that $\mathcal{G}_{x}$ is a Cohen-Macaulay $\mathcal{O}_{\mathbb{P}^n,x}$-module for each $x\in \Supp(\mathcal{G})$ by our hypothesis, so we claim that
$$
N=\bigoplus_{m\in \mathbb{Z}}H^0\big(\mathbb{P}^n,\mathcal{G}\otimes\mathcal{O}_{\mathbb{P}^n}(m)\big)
$$
has finite length local cohomologies supported at the irrelevant ideal of $k[x_0,\ldots,x_n]$ (excluding the top local cohomology). Indeed for any homogeneous prime ideal $\fp \subset k[x_0,\ldots,x_n]$ not containing $(x_0,\ldots,x_n)$, we have $N_{(\fp)}=(N_{(\fp)})_{[0]}[T,T^{-1}]$ for some degree $1$ invertible element $T \in k[x_0,\ldots,x_n]_{(\fp)}$. The Cohen-Macaulay property of $\mathcal{G}_{x}$ implies the same property on $(N_{(\fp)})_{[0]}$. Thus $N_{(\fp)}$ and so $N_\fp$ are Cohen-Macaulay. This observation together with 
\cite[Graded Annihilator and Finiteness Theorem 14.3.10]{BrodmannSharpLocal}
shows that  $N$ is a generalized Cohen-Macaulay module over $k[x_0,\ldots,x_n]/0:_{k[\mathbf{x}]}N$.
So the desired finite length property holds by \cite[4.2.1 Independence Theorem]{BrodmannSharpLocal}. From Remark \ref{RemarkProjToCone}\ref{itm:LocalCohomolgyVSSheafCohomolgy} and \ref{itm:DepthOfAssociatedGradedModule}, we get $$H^0_{R(X,D)_+}\big(\Gamma_*(X,\mathcal{F},D)\big)=0=H^1_{R(X,D))_+}\big(\Gamma_*(X,\mathcal{F},D)\big)$$ and also we get
$$
H^{i+1}_{R(X,D)_{+}}\big(\Gamma_*(X,\mathcal{F},D)\big)_{[m]}=H^i\big(X,\mathcal{F}\otimes \mathcal{O}_{X}(mD)\big)=H^i\big(\mathbb{P}^n,\mathcal{G}\otimes \mathcal{O}_{\mathbb{P}^n}(m)\big)=\big(H^{i+1}_{(x_0,\ldots,x_n)}(N)\big)_{[m]}=0
$$
for $|m|\gg 0$ and $1\le i<\dim (X)$. So $\Gamma_*(X,\mathcal{F},D)$ is generalized Cohen-Macaulay.
	
\underline{Step $2$: the general case.} Pick some $N\in\mathbb{N}$ such that $ND$ is a very ample integral Cartier divisor. Then for each $i\in \mathbb{Z}$, the coherent sheaf $\mathcal{F}\otimes \mathcal{O}_X(\lfloor iD\rfloor)$ is  maximal Cohen-Macaulay, because  so is $\mathcal{F}$ and  $\mathcal{O}_X(\lfloor iD\rfloor)$ is a line bundle by our assumption. Thus applying Step $1$ to each of the maximal Cohen-Macaulay sheaves $\mathcal{F}\otimes \mathcal{O}_X(\lfloor iD\rfloor)$ for $0\le i\le N-1$ and to the fixed very ample integral Cartier divisor $ND$, there is a unified choice of $m$ with
$$
H^j_{R(X,D)_+}\big(\Gamma_*(X,\mathcal{F},D)\big)_{[i+Nk]}=H^{j-1}\big(X,\mathcal{F}\otimes \mathcal{O}_X\big(\lfloor (i+kN)D\rfloor \big)\big)
$$
$$
=H^{j-1}\big(X,(\mathcal{F}\otimes \mathcal{O}_X(\lfloor iD\rfloor)) \otimes \mathcal{O}_X(kND)\big)=0
$$
for $|k|\ge m$, $1-N\le i\le N-1$ and $2\le j\le\dim X$  (the first equality in this  display holds in view of Remark \ref{RemarkProjToCone}\ref{itm:LocalCohomolgyVSSheafCohomolgy}). From these vanishings, for all but finitely many components as well as $\depth\big(\Gamma_*(X,\mathcal{F},D)\big)\ge 2$ (Remark \ref{RemarkProjToCone}\ref{itm:DepthOfAssociatedGradedModule}), we deduce our desired conclusion that $\Gamma_*(X,\mathcal{F},D)$ is generalized Cohen-Macaulay.
\end{proof}

Now we show that the existence of a maximal Cohen-Macaulay module over an $\mathbb{N}_0$-graded normal domain $R$ whose degree $0$ part is an $F$-finite field depends on the existence of a maximal Cohen-Macaulay sheaf on $X=\Proj(R)$, provided $X$ is locally factorial. This follows from the following corollary in conjunction with Remark \ref{RemarkProjToCone}\ref{itm:Demazure}.  In particular, if some generalized section ring of an ample $\mathbb{Q}$-divisor on $X$ admits a maximal Cohen-Macaulay module, then so does every such a generalized section ring. The dimension restriction in the statement of the next lemma is assumed only because instances of maximal Cohen-Macaulay modules  over prime characteristic $\mathbb{N}_0$-graded domains of dimension $\le 3$ are already known.

\begin{corollary}
\label{CorollaryProjectiveVariety}
Let $X$ be a normal projective variety of dimension $\ge 3$ over an $F$-finite field $k$ and let $D$ be an ample $\mathbb{Q}$-Weil divisor. Assume that $\mathcal{O}_X(\lfloor iD\rfloor)$ is an invertible sheaf for each $i\in \mathbb{Z}$, which is the case whenever $X$ is locally factorial, or $D$ is integral (and Cartier) or $R(X,D)$ satisfies the Goto-Watanabe condition $(\#)$.

\begin{enumerate} [(i)]
\item
\label{itm:MCMSheafEquivalentToExistenceOfACM}
$R(X,D)$ admits a graded maximal Cohen-Macaulay module if and only if there is a  maximal Cohen-Macaulay sheaf on $X$.

\item
\label{itm:LocallyFactorialIsNeedForEquivalenceSomeandAll}
If $X$ is locally factorial, then \label{itm:ACMisIndependeOfDivisor}  $R(X,D)$ admits a maximal Cohen-Macaulay module provided the  generalized section ring of any ample $\mathbb{Q}$-divisor over $X$ admits a maximal Cohen-Macaulay module.
\end{enumerate}
\end{corollary}

\begin{proof}
\ref{itm:MCMSheafEquivalentToExistenceOfACM}: Although one side is a folklore fact, we illustrate it for the convenience of the reader (the $\mathbb{Q}$-divisor case is less treated in the literature).  Let $M$ be a graded maximal Cohen-Macaulay $R(X,D)$-module, and let $N$ be a positive integer such that $ND$ is very ample so that $R(X,ND)=R(X,D)^{(N)}$ is a standard graded subring of $R(X,D)$. By our hypothesis, $M$ is also a maximal Cohen-Macaulay $R(X,ND)$-module (Remark \ref{RemarkProjToCone} \ref{itm:FactVeroneseModulesFiniteOverVeronesSubring}) and so is its graded $R(X,ND)$-direct summand $M^{(N)}$. There is an open cover $\{D_+(f_i)\}_{i=1}^m$ of $X$, where $f_1,\ldots,f_m$ is a  $k$-vector space basis  of $R(X,ND)_{[1]}=R(X,D)_{[N]}$. Since ${\big(M^{(N)}_{f_i}\big)}_{[0]}={\big(M_{f_i}\big)}_{[0]}$ for each $1\le i\le m$, so $\widetilde{M}=\widetilde{M^{(N)}}$ are the same coherent sheaves on $X=\Proj \big(R(X,D)\big)=\Proj \big(R(X,ND)\big)$. Consequently, without loss of generality, we can assume that $D$ is very ample and $R(X,D)$ is standard. In this case, we have $M_{(\fp)}=(M_{(\fp)})_{[0]}[T,T^{-1}]$ and $R_{(\fp)}=(R_{(\fp)})_{[0]}[T,T^{-1}]$ for some variable $T$ of degree $1$ and each $\fp\in X$ and so the statement holds.

Conversely suppose that $\mathcal{F}$ is a maximal Cohen-Macaulay sheaf on $X$ and consider the $\mathbb{Z}$-graded $R(X,D)$-module $\Gamma_*(X,\mathcal{F},D)$ which is a generalized Cohen-Macaulay $R(X,D)$-module by Lemma \ref{LemmaMCMSheafToModule}. Then our statement follows from Proposition \ref{PropositionAlternativeProof1}, as $\depth_R\big(\Gamma_*(X,\mathcal{F},D)\big)\ge 2$ by Remark \ref{RemarkProjToCone}\ref{itm:DepthOfAssociatedGradedModule}.

\ref{itm:LocallyFactorialIsNeedForEquivalenceSomeandAll}: This part is an immediate consequence of \ref{itm:MCMSheafEquivalentToExistenceOfACM}.
\end{proof}

A sufficient condition involved in the statement of Corollary \ref{CorollaryProjectiveVariety} was the case where $X$ is locally factorial. This condition is satisfied for instance if $R$ is a factorial domain whose defining $\mathbb{Q}$-divisor $D$ satisfies a further condition:

\begin{lemma}
\label{LemmaFactorialImpliesLocallyFactorial}
Let  $R$ be a factorial $\mathbb{N}_0$-graded domain  such that $R_{[0]}$ is a field of arbitrary characteristic. Let $D$ be the defining ample $\mathbb{Q}$-divisor of $R$ on $X:=\Proj(R)$, i.e. $R=R(X,D)$ (see Remark \ref{RemarkProjToCone}\ref{itm:Demazure}). Let $L$ be the least common multiple of the denominators $q_V$ appearing in the presentation $D=\sum\limits_{\text{prime divisor} \ V} (p_V/q_V) V$, where for each $V$ the two integers  $p_V,q_V$ are coprime and $q_V\ge 1$. Then the following statements hold.

\begin{enumerate}[(i)]
\item\label{itm:ProjLocallyFactorial}
$\Proj(R)$ is almost locally factorial in the sense that each of its local rings has torsion divisor class group.

\item\label{itm:ProjFactorial}
$X$ is locally factorial if and only if the integral divisor  $LD$ is a Cartier divisor.

\item\label{itm:ProjFactorialityCounterexample}
If $R=K[x,y,z]$ where $x$ (respectively,  $y$ and $z$) has degree $2$ (respectively,  $3$ and $4$), then $\Proj(R)$ is almost locally factorial  but it is not locally factorial.
\end{enumerate}
\end{lemma}

\begin{proof} 
\ref{itm:ProjLocallyFactorial}: Since $R(X,D)$ is factorial by our hypothesis,   it turns out from \cite[Corollary 1.7 and Theorem 1.6]{WatanabeSomeRemarks} that $\Cl(X)=\{kLD:k\in \mathbb{Z}\}$. Since some multiple $ND$ of $D$ generates an (ample) invertible sheaf as the $\mathbb{Q}$-divisor $D$ is ample, hence $LND$ is Cartier and it follows that $\Cl(X)/\Pic(X)$ is torsion. Let $x \in X$ be any point and let $U:=\Spec(A) \subset X$ be an affine open neighborhood of $x \in X$. Set
$A_x:=\mathcal{O}_{X,x}$. Then we want to show that $\Cl(A_x)$ is torsion. To this aim, letting $\fp \subset A_x$ be a prime ideal of height one, we will show that there is some $N>0$ such that the reflexive closure of $\fp^{N}$ is principal by the definition of the divisor class group. Lift $\fp$ to a prime ideal $\fq \subset A$ via the localization map. So $\fq$ defines a Weil prime divisor $V(\fq) \subset U$. Denote by $D'$ the Zariski closure of $V(\fq)$ in $X$. So $D'$ defines a Weil divisor on $X$ (see \cite[Proposition 11.40]{GortzWedhornAlgebraic}, or \cite[II, Proposition 6.5(a)]{HartshorneAlgebraicGeometry}). Since the quotient group $\Cl(X)/\Pic(X)$ is torsion, there exists $N>0$ such that $ND'$ is locally principal, which gives that the reflexive closure of $\fp^{N}$ is principal. Hence  we proved that $\Cl(A_x)$ is torsion, as desired.\footnote{In \cite[Corollaire (21.6.10)]{GrothendieckEGAIV2}, it is proved that $\Cl(X)/\Pic(X)$ is trivial if and only if $X$ is locally factorial.}

\ref{itm:ProjFactorial}: Suppose that the integral Weil divisor $LD$ is Cartier. Then we get $\Cl(X)=\Pic(X)=\mathbb{Z}[LD]$ by the proof of part (i). We can repeat the same argument of (i) to show that $\Cl(\mathcal{O}_{X,x})=0$ for every $x \in X$. Hence $X$ is locally factorial. Conversely, suppose that $X$ is locally factorial. Then the statement holds in view of \cite[II, Proposition 6.11]{HartshorneAlgebraicGeometry}.



\ref{itm:ProjFactorialityCounterexample}: It is readily seen that the affine open subset $(R_Z)_{[0]}$ of $X=\Proj(R)$ is
$$
K[x^2/z,y^4/z^3,xy^2/z^2]\cong K[a,b,c]/(ab-c^2),
$$
which is not factorial (see \cite[Example 2.3]{GriffithSomeResults}). Alternatively, we can appeal to \cite[Theorem 2.7]{RobbianoFactorial} to conclude that $X$ is not locally factorial.
\end{proof}

\begin{remark}
\label{RemarkTavanfarReduction}
In \cite{TavanfarReduction}, it is proved that the small Cohen-Macaulay conjecture can be reduced to the case of existence of maximal Cohen-Macaulay modules over some excellent (incomplete) factorial domain that is a homomorphic image of a regular local ring. Using a graded analogue of the methods used in the proof of \cite{TavanfarReduction}, one can see that the existence of a maximal Cohen-Macaulay module over an $\mathbb{N}_0$-graded normal domain $R$, with  $R_{[0]}$ an $F$-finite field, reduces to the existence of a maximal Cohen-Macaulay module over an $\mathbb{Z}$-graded factorial domain $U$ that is possibly not  positively graded  with $U_{[0]}$ a field (depending on the degrees of the generators of the prime almost complete intersection ideal involved in the construction). The case where $U$ is $\mathbb{N}_0$-graded and $U_{[0]}$ is a field  can be useful for   relaxing the factorial condition in the statement of  Corollary \ref{CorollaryProjectiveVariety}. At the time of writing this paper, we do not know how to reduce ourselves to the case of such $\mathbb{N}_0$-graded factorial domains in general.
\end{remark}

One may ask whether the converse of Proposition \ref{PropositionAlternativeProof1} holds, i.e. whether  Cohen-Macaulayness of some  $M^{(p^e,j)}$ forces  $M$ to be generalized Cohen-Macaulay. This question is considerable from the small Cohen-Macaulay conjecture point of view, because an affirmative answer makes no chance for any $R^{(p^e,j)}$ to be Cohen-Macaulay provided $R$ is not generalized Cohen-Macaulay. As a second application of Lemma \ref{LemmaMCMSheafToModule}, one observes that the converse of proposition \ref{PropositionAlternativeProof1} holds under a ubiquitous condition.  
However,  Corollary \ref{CorollaryVeroneseCMImpliesGeneralizedCohenMacaulay} does not say anything about the case of a generalized section ring of a non-integral ample $\mathbb{Q}$-divisor  even over a locally factorial projective variety.

\begin{corollary}
\label{CorollaryVeroneseCMImpliesGeneralizedCohenMacaulay}
Let $R$ be a Noetherian normal  $\mathbb{N}_0$-graded domain of dimension $\ge 2$, where $R_{[0]}$ is an $F$-finite field and let $M$ be a finitely generated $\mathbb{Z}$-graded module of depth $\ge 2$. Suppose that  $R$  satisfies the Goto-Watanabe condition $(\#)$  (e.g. $R$ is standard or is the section ring of an ample integral (Cartier) divisor on $\Proj(R)$). If the Veronese module  $M^{(p^e,j)}$ is Cohen-Macaulay for a single choice $(e,j)$ such that $e\ge 1$ and $0\le j\le p^e-1$, then $M$ is  generalized Cohen-Macaulay.
\end{corollary}

\begin{proof}
Since we have $M^{(p^e,j)}=M(j)^{(p^e)}$, after replacing $M$ with $M(j)$ if necessary, we can assume that $M^{(p^e)}$ is maximal Cohen-Macaulay. Let $f_1,\ldots,f_n$ be a sequence of homogeneous elements of $R$ all of the same degree $kp^e$ for some fixed $k,e\in \mathbb{N}$, which generate an ideal primary to $R_+$. By our hypothesis, $M^{(p^e)}$ is a maximal Cohen-Macaulay $R^{(p^e)}$-module (Remark \ref{RemarkDifferentPerspective}\ref{itm:VeroneseSubmoduleCM}). Since $\big(M^{(p^e)}_{f_i}\big)_{[0]}=(M_{f_i})_{[0]}$ for each $1\le i\le n$, $\widetilde{M^{(p^e)}}=\widetilde{M}$ is a maximal Cohen-Macaulay sheaf on $X:=\Proj(R^{(p^e)})=\Proj(R)$ (see the first paragraph of the proof of Corollary \ref{CorollaryProjectiveVariety}).  Thus since $\mathcal{O}_X(n)$ is an invertible sheaf for each $n\in \mathbb{Z}$ by our hypothesis and \cite[Lemma 5.1.2]{WatanabeSomeRemarks}, we conclude from Lemma \ref{LemmaMCMSheafToModule} that $\Gamma_*(X,\widetilde{M})=\Gamma_*\big(X,\widetilde{M},\mathcal{O}_X(1)\big)$ is a generalized Cohen-Macaulay $R$-module (recall that $R=R\big(X,\mathcal{O}_X(1)\big)$ by the normality of $R$). Finally, from the exact sequence $0\rightarrow \Gamma_{R_+}(M)\rightarrow M\rightarrow \Gamma_*(X,\widetilde{M})\rightarrow H^1_{R_+}(M)\rightarrow 0$ by \cite[5.1.6(i)]{GotoWatanabeOnGraded} as well as our depth hypothesis on $M$, we get $M\cong \Gamma_*(X,\widetilde{M})$ and so the corollary follows.
\end{proof}

In spite of Corollary \ref{CorollaryVeroneseCMImpliesGeneralizedCohenMacaulay}, it is still possible that $F^e_*(R)$ admits a Cohen-Macaulay direct summand even when $R$ is not generalized Cohen-Macaulay. However, this Cohen-Macaulay direct summand is not expected to be some Veronese submodule $R^{(p^e,j)}$ anymore because of Corollary \ref{CorollaryVeroneseCMImpliesGeneralizedCohenMacaulay} (if $R$ is not generalized Cohen-Macaulay). We give an example which can arise as a polynomial extension of a graded generalized Cohen-Macaulay domain (such a polynomial extension is not generalized Cohen-Macaulay). To this aim, we need the following new definition.

\begin{definition}
\label{DefinitionDirectSummandOfVeroneseSubmodule}
Let $C$ be a Noetherian $\mathbb{Z}$-graded ring and let $\mathbf{X}:=X_1,\ldots,X_c$ be a sequence of indeterminates over $C$ such that degree$(X_i)=h_i$,  where $h_i$ is possibly zero. Consider $C[\mathbf{X}]$ as an $\mathbb{Z}$-graded ring via the total degree obtained by sum of  degrees of elements appearing in each monomial. Fix some $e\in \mathbb{N},\ 0\le i\le p^e-1$ and $\underline{t}=(t_1,\ldots,t_c)\in \mathbb{N}_0^c$ such that $0\le t_1,\ldots,t_c\le p^e-1$.  Then, we define $C[\mathbf{X}]^{(p^e,i,\underline{t})}$ as 
$$
C[\mathbf{X}]^{(p^e,i,\underline{t})}=\bigoplus\limits_{k\in \mathbb{Z}}\ \Big(\bigoplus\limits_{\underline{n}\in \mathbb{N}_0^c}\ {C}_{[(k-\sum\limits_{u=1}^cn_uh_u)p^e+i-\sum\limits_{u=1}^ct_uh_u]} X_1^{n_1p^e+t_1}\cdots X_c^{n_cp^e+t_c}\ \Big).
$$
\end{definition}

\begin{lemma}
\label{LemmaDirectSummandOFVeronese}
With the notation and hypothesis of Definition \ref{DefinitionDirectSummandOfVeroneseSubmodule}, the following statements hold.
\begin{enumerate}[(i)]
\item
\label{itm:DirectSummandOFVeroneseIsModule}
$C[\mathbf{X}]^{(p^e,i,\underline{t})}$ admits a $\frac{1}{p^e}\mathbb{Z}$-graded $C[\mathbf{X}]$-module whose $\frac{h}{p^e}$-th component with $h=kp^e+j$ is
$$
\begin{cases}0,& j\neq i\\\bigoplus\limits_{\underline{n}\in \mathbb{N}_0^c}\ {C}_{[(k-\sum\limits_{u=1}^cn_uh_u)p^e+i-\sum\limits_{u=1}^ct_uh_u]} X_1^{n_1p^e+t_1}\cdots X_c^{n_cp^e+t_c}\ \subseteq C[\mathbf{X}]_{[h]},& j=i\end{cases}.
$$

\item
\label{itm:DirectSummandOfVeroneseIsDirectSummandOFVeronese}
The Veronese module $C[\mathbf{X}]^{(p^e,i)}$ as a graded $C[\mathbf{X}]$-module admits the graded decomposition
$$
C[\mathbf{X}]^{(p^e,i)}\cong\bigoplus\limits_{\substack{\underline{t}\in \mathbb{N}_0^c \\ 0\le t_1,\ldots,t_c\le p^e-1}} C[\mathbf{X}]^{(p^e,i,\underline{t})}.
$$

\item
\label{itm:LiftingTheVeroneseSubmodule}
Let $h_u=1$ for each $1\le u \le c$. For some fixed $0\le i\le p^e-1$ and $0\le j\le i$, we set $t:=(0,\ldots,0,\underset{}{j},0,\ldots,0)\in \mathbb{N}_0^c$ where $j$ appears at some (arbitrary) $u$-th spot.  Then
$$
C[\mathbf{X}]^{(p^e,i,\underline{t})}/\mathbf{X}C[\mathbf{X}]^{(p^e,i,\underline{t})}\cong C^{(p^e,i-j)}~\mbox{as}~C\mbox{-modules}.
$$
\end{enumerate}
\end{lemma}

\begin{proof}
The proof is easy and is left to the reader.
\end{proof}

\begin{corollary}\label{CorollaryTrivialDeformation}
Let $R$ be a generalized Cohen-Macaulay $\mathbb{Z}$-graded $F$-finite local domain. Consider the polynomial extension $A:=R[\mathbf{X}]$ of $R$ at finitely many variables. Then $F^e_*(A)$ comprises a maximal Cohen-Macaulay direct summand.
\end{corollary}

\begin{proof}
By Proposition \ref{PropositionAlternativeProof1}, $R^{(p^e,i)}$ is a maximal Cohen-Macaulay module for some $e,i$. So in view of Lemma \ref{LemmaDirectSummandOFVeronese}\ref{itm:LiftingTheVeroneseSubmodule}, it suffices to notice that $\mathbf{X}$ is a regular sequence on $R[\mathbf{X}]^{(p^e,i,\underline{0})}$ which is immediate as  $R[\mathbf{X}]^{(p^e,i,\underline{0})}$ is a direct summand of $F^e_*(A)$ in view of Lemma \ref{LemmaDirectSummandOFVeronese}\ref{itm:DirectSummandOfVeroneseIsDirectSummandOFVeronese} and Remark \ref{RemarkDifferentPerspective}\ref{itm:VeroneseSubmoduleDirectSummandOfFrobeniusDirectImage}.
\end{proof}

At the time of writing the paper, the authors do not know whether an analogue of Corollary \ref{CorollaryTrivialDeformation} is imaginable/provable for a non-trivial deformation of a generalized  Cohen-Macaulay graded local domain.

\subsection{Cohen-Macaulay direct summands of Frobenius direct images}

We continue by presenting additional instances of graded local rings admitting a maximal Cohen-Macaulay module. First, we have to return to the local case. In \cite{SchoutensHochster}, it is proved that an $F$-split complete local ring admits a maximal Cohen-Macaulay module if it has dimension $3$. In higher dimension,  as shown in the next lemma, the same result holds for an $F$-pure (complete) ring  provided it is generalized Cohen-Macaulay. The difference is that in dimension $3$, $F$-split rings may not be generalized Cohen-Macaulay and so the resulting maximal Cohen-Macaulay module occurs as the canonical module $\omega_M$ of $M$, where $M$ is the complement of $R$ as a direct summand in $F_*(R)$. If $R$ is $F$-pure and generalized Cohen-Macaulay, the complement $M$ of $R$ is maximal Cohen-Macaulay (in fact, the proof of \cite{SchoutensHochster} is more complicated than the proof of Lemma \ref{LemmaFPureGCMLocal}\ref{itm:FPureGCMLocalPerfectRF}, because one has to prove the Cohen-Macaulayness of $\omega_M$ which requires more work):

\begin{lemma}
\label{LemmaFPureGCMLocal}
Let $(R,\fm)$ be an $F$-finite $F$-split (equivalently, $F$-pure) generalized Cohen-Macaulay local ring and let $F_*(R)=R\oplus M$.

\begin{enumerate}[(i)]
\item
\label{itm:FPureGCMLocalPerfectRF}
If $R/\fm$ is perfect, then $M$ is a maximal Cohen-Macaulay $R$-module.	

\item
\label{itm:FPureGCMLocalGeneral}
If $R/\fm$ is imperfect  but $R$ has a coefficient field (e.g. if $R$ is complete), then $R$ admits a maximal Cohen-Macaulay $R$-module (which is possibly not $M$).
\end{enumerate}
\end{lemma}

\begin{proof}
\ref{itm:FPureGCMLocalPerfectRF}: The decomposition $F_*(R)=R\oplus M$ yields \begin{equation}
\label{length1}
F_*\big(H^i_{\fm}(R)\big)\cong H^i_{\fm}\big(F_*(R)\big)=H^i_{\fm}(R)\oplus H^i_{\fm}(M)
\end{equation}
for each $i\in \mathbb{N}_0$. On the other hand, in view of  perfectness of $R/\fm$, one can readily see  that  $\ell\Big(F_*\big(H^i_{\fm}(R)\big)\Big)=\ell\big(H^i_{\fm}(R)\big)$ for all $i$, thus by applying the length function to $(\ref{length1})$, we get $\ell\big(H^i_{\fm}(M)\big)=0$ for each $0\le i\le \dim(R)-1$.

\ref{itm:FPureGCMLocalGeneral}: By our hypothesis, we have a coefficient field $K$ of $R$. Then $\fm_\infty:=\fm(R\otimes_KK^{\infty})$ is a maximal ideal of $R\otimes_KK^{\infty}$ and $(R\otimes_KK^\infty)_{\fm_\infty}$ is Noetherian by \cite[Proposition 2.2.12]{TavanfarTest}. Consequently, $R \rightarrow (R\otimes_KK^\infty)_{\fm_\infty}$ is a flat weakly unramified local homomorphism of Noetherian local rings whose target has the perfect residue field $K^\infty$.  It turns out that $R_{K^\infty}:=(R\otimes_{K}K^\infty)_{\fm_\infty}$ is also $F$-pure by virtue of \cite[Proposition 3.3(1)]{AberbachExtension}. Thus from part \ref{itm:FPureGCMLocalPerfectRF}, we conclude that $M_{K^\infty}$ is a maximal Cohen-Macaulay $R_{K^\infty}$-module, where $M_{K^\infty}$ is the complement of $R_{K^\infty}$ in $F^e_*(R_{K^\infty})$ as a direct summand. Let $H$ be a matrix of elements of $\fm R_{K^\infty}$ that are involved in the presentation of $M_{K^\infty}$ by finite free $R_{K^\infty}$-modules. There are only finitely many scalars in $K^\infty$ that appear in $H$, and after  adjoining them to $K$, we get a finite purely inseparable extension $L/K$, $\fm_L:=\fm(R\otimes_K L)$, and an $(R\otimes_KL)_{\fm_L}$-module $N$ such that $N \otimes_{(R\otimes_KL)_{\fm_L}}(R\otimes_KK^{\infty})_{\fm_\infty}\cong M_{K^\infty}$. Consequently, $N$ is maximal Cohen-Macaulay as an $(R\otimes_KL)_{\fm_L}$-module and as an $R$-module. 
\end{proof}

For the graded version of the previous lemma,  we are able to decrease the codimension of the non-$F$-pure locus and the non-Cohen-Macaulay locus.

\begin{corollary}
\label{CorollaryFPureGraded}
Let $R$ be an  $\mathbb{N}_0$-graded normal domain over an $F$-finite field $K:=R_{[0]}$ with algebraic closure $K^{\alg}$ such that the following hold.

\begin{enumerate}[(i)]
\item
All local rings of $\Proj(R\otimes_KK^{\alg})$ are $F$-pure and generalized Cohen-Macaulay (e.g. $R\otimes_KK^{\alg}$ is $F$-pure and generalized Cohen-Macaulay at each homogeneous non-maximal prime ideal).

\item
Either $\Proj (R)$ is locally factorial (e.g. $R$ is factorial as in Lemma \ref{LemmaFactorialImpliesLocallyFactorial}\ref{itm:ProjFactorial}), or $R$ satisfies the Goto-Watanabe condition $(\#)$ (e.g. $R$ is standard or the section ring of an integral ample (Cartier) divisor).
\end{enumerate} 
Then $R$ admits a $\mathbb{Z}$-graded maximal Cohen-Macaulay module.
\end{corollary}

\begin{proof}
Let us show that $X:=\Proj (R)$ admits a maximal Cohen-Macaulay sheaf. Then our statement follows from Corollary \ref{CorollaryProjectiveVariety}. Without loss of generality, we can assume that $R$ has dimension $\ge 4$.

For any $K$-scheme $Y$ and any field extension $F/K$, we denote the base change $Y\times_KF$ by the abbreviated notation $Y_{F}$. Indeed,  $X_{K^{\alg}}=\Proj (R\otimes_KK^{\alg})$. Let $\{U_i\}_{i=1}^n$ be an affine cover of $X$ and set $\mathcal{G}:=F_*(\mathcal{O}_{X_{K^{\text{alg}}}})/\mathcal{O}_{X_{K^{\alg}}}$.   In view of Lemma \ref{LemmaFPureGCMLocal}\ref{itm:FPureGCMLocalPerfectRF}, $\mathcal{G}|_{{U_i}_{K^{\alg}}}=F_*({\mathcal{O}_{U_i}}_{{K^{\alg}}})/\mathcal{O}_{{U_i}_{{K^{\text{alg}}}}}$ is a maximal Cohen-Macaulay sheaf on the affine open subset ${U_i}_{K^{\text{alg}}}$ of $X_{K^{\text{alg}}}$ for each $1\le i\le n$, because the localization of the finitely generated $K^{\text{alg}}$-algebra ${U_i}_{K^{\alg}}$ at any closed point of ${U_i}_{K^{\alg}}$ has the perfect residue field $K^{\alg}$. Thus $\mathcal{G}$ is a maximal Cohen-Macaulay $\mathcal{O}_{X_{K^{\alg}}}$-module. Let
$$
\bigoplus_{k=1}^s\mathcal{O}_{X_{K^{\alg}}}(n_k)\overset{H}{\rightarrow} \bigoplus_{k=1}^{t}\mathcal{O}_{X_{K^{\alg}}}(m_k)\rightarrow \mathcal{G}\rightarrow 0
$$
be a presentation of $\mathcal{G}$ (\cite[Corollary 5.18 and Proposition 5.7]{HartshorneAlgebraicGeometry}). There is a finite extension $L$ of $K$ containing all coefficients in $K^{\text{alg}}$ appearing in the matrix of some $H|_{{U_i}_{K^{\alg}}}$ for any $1\le i \le n$ (we are assuming the twisted structure sheaves are free on the affine cover). Our choice of $L$ enables us to consider the coherent $\mathcal{O}_{X_L}$-module $\mathcal{G}_{L}$ as the cokernel of
$$
\bigoplus_{k=1}^s\mathcal{O}_{X_L}(n_k)\overset{H_L}{\rightarrow} \bigoplus_{k=1}^{t}\mathcal{O}_{X_L}(m_k),
$$
where $H_L$ is defined by glueing via $H_L|_{{U_i}_L}=H|_{{U_i}_{K^{\alg}}}$. From  flatness of $\pi:X_{K^{\alg}}\rightarrow X_{L}$ as well as the identity $\mathcal{G}=\pi^*(\mathcal{G}_L)$, we deduce that $\mathcal{G}_L$ is a maximal Cohen-Macaulay $\mathcal{O}_{X_L}$-module. Finally, as $\pi':X_{L}\rightarrow X$ is a finite morphism,  ${\pi'}_*(\mathcal{G}_{L})$ is a maximal Cohen-Macaulay $\mathcal{O}_X$-module.
\end{proof}

\begin{remark}
\label{RemarkAltration}
Let $R$ be a normal graded ring  as in the statement of Corollary \ref{CorollaryFPureGraded} such that $X:=\Proj (R)$ is not necessarily $F$-pure or generalized Cohen-Macaulay. Let $\pi:Y\rightarrow X:=\Proj (R)$ be a regular alteration of  $X$ ($Y$ is a nonsingular projective variety). We recall that the pushforward along $\pi$ preserves coherence, because $\pi$ is proper. The question on the existence of  coherent $\mathcal{O}_Y$-modules whose pushforward along $\pi$ are maximal Cohen-Macaulay $\mathcal{O}_X$-modules does not seem to be studied in the literature, although an affine version of  such an  example can be found in the proof of \cite[Proposition 4.11]{SchoutensHochster}.
A similar intriguing result has been established originally by Roberts (\cite{RobertsCohenMacaulay}) in equal characteristic zero in order to provide an alternative proof for the New Intersection Theorem in equal characteristic zero without reduction to prime characteristic: the derived pushforward of the structure sheaf of a resolution of singularity of a local ring $S$ of a projective variety over a field of characteristic zero is a \textit{maximal Cohen-Macaulay complex} of $S$-modules. This was improved later in the recent work of Iyengar, Ma, Schwede and Walker (\cite{IyengarMaEtAl}), where the authors show that the ``dualizing complex dual of the canonical module" provides another instance of a maximal Cohen-Macaulay complex (over any local ring admitting a dualizing complex). While maximal Cohen-Macaulay complexes can successfully be applied to take on some roles which maximal Cohen-Macaulay modules play, it is not always the case that maximal Cohen-Macaulay complexes can be a substitute for maximal Cohen-Macaulay modules (see e.g. \cite[Remark 3.17]{IyengarMaEtAl}).
\end{remark}

Now we examine the reason of assuming the perfectness on the residue field in the statement of Theorem \ref{FactNonGradedPrefectResidueField}. It turns out that while the perfectness assumption was removable in  Proposition \ref{PropositionAlternativeProof1} in the setting of graded rings and Veronese submodules,  it is essential in the local setting, as the following proposition shows.

\begin{proposition}
\label{PropositionGradedTrueNonGradedWrong}
There exists a local (non-graded) $F$-finite domain $(R,\fm)$, together with a torsion-free generalized Cohen-Macaulay $R$module $M$, such that the following hold.
\begin{enumerate}
\item
$F^e_*(M)$ decomposes into $n_e$-number of non-zero $R$-direct summands $M_{e,0},\ldots,M_{e,n_e}$ with $\lim\limits_{e\rightarrow \infty}n_e=\infty $.

\item
No direct summand $M_{e,i}$ of $F_{*}^{e}(M)$ is a maximal Cohen-Macaulay module.
\end{enumerate}
\end{proposition}

\begin{proof}
Suppose the contrary to get a contradiction. Let $R$ be a normal $\mathbb{N}_0$-graded local domain such that $R_{[0]}$ is an $F$-finite field and the $R$-modules $R^{(p^{e},i)}$ are all non-Cohen-Macaulay.  For instance, in view of Corollary \ref{CorollaryVeroneseCMImpliesGeneralizedCohenMacaulay}, we can take $R$ as any non-generalized Cohen-Macaulay standard  $\mathbb{N}_0$-graded normal domain with $R_{[0]}$ an $F$-finite field. Then the localizations of $R$ are $F$-finite. Moreover, we have the decompositions
$$
F_{*}^{e}(R)=\bigoplus_{i=0}^{p^{e}-1}R^{(p^{e},i)},\ \ \ \ F_{*}^{e'}(R^{(p^{e},i)})=\bigoplus_{j=0}^{p^{e'}-1}R^{(p^{e+e'},jp^{e}+i)}
$$
(see Remark \ref{RemarkDifferentPerspective}
\ref{itm:VeroneseSubmoduleDirectSummandOfFrobeniusDirectImage} and \ref{itm:FrobeniusDirectImageOfVeronseIsDirctSumOfVeronese}). Notice that the support of each non-zero direct summand $R^{(p^{e},i)}$ coincides with $\Spec(R)$, because $R$ is assumed to be a domain. Since $R$ is a finitely generated over a field, the non-Cohen-Macaulay locus of any finitely generated $R$-module is Zariski-closed. Let the symbol non-CM$(M)$ denote the non-Cohen-Macaulay locus of a finitely generated $R$-module $M$. Let $\fp$ be a minimal prime ideal of non-CM$(R)$. Then $R_\mathfrak{p}$ is generalized Cohen-Macaulay by \cite[9.5.1 Faltings' Annihilator Theorem]{BrodmannSharpLocal} and $F_{*}^{e}(R_\fp)=\big(F_{*}^{e}(R)\big)_{\mathfrak{p}}=\bigoplus_{i=0}^{p^{e}}(R^{(p^{e},i)})_{\mathfrak{p}}$ decomposes into strictly increasing number of non-zero direct summands tending to infinity  as $e$ increases (see Remark \ref{RemarkDifferentPerspective}
\ref{itm:MTorsionFreeImpliesGoesToInfinity}). In particular by our hypothesis, some $(R^{(p^{e_{1}},i_{1})})_{\mathfrak{p}}$ is Cohen-Macaulay. Now let $\mathfrak{q}$ be a minimal member of $$\text{non-CM}(R^{(p^{e_{1}},i_{1})})\subseteq\text{non-CM}\big(F_{*}^{e_{1}}(R)\big)\backslash\{\fp\}=\text{non-CM}(R)\backslash\{\fp\}.$$ Then $(R^{(p^{e_{i}},i_{1})})_{\fq}$ is a generalized-Cohen-Macaulay module over $R_{\fq}$ whose Frobenius direct images decompose as 
$$
F_{*}^{e'}\big((R^{(p^{e_{1}},i_{1})})_{\fq}\big)=\big(F_{*}^{e'}(R^{(p^{e_{1}},i_{1})})\big)_{\fq}=\bigoplus_{j=0}^{p^{e'}-1}(R^{(p^{e_{1}+e'},jp^{e_{1}}+i_{1})}){}_{\fq}=\bigoplus_{j=0}^{p^{e'}-1}\big(\big(R^{(p^{e_{1}},i_{1})})^{(p^{e'},j)}\big){}_{\fq}
$$
in view of Remark \ref{RemarkDifferentPerspective}\ref{itm:FrobeniusDirectImageOfVeronseIsDirctSumOfVeronese}. Thus by generalized Cohen-Macaulay property and Remark \ref{RemarkDifferentPerspective}\ref{itm:MTorsionFreeImpliesGoesToInfinity} as well as our hypothesis, some
$$
(R^{(p^{e_{2}:=e_{1}+e'},\ i_{2}:=jp^{e_{1}}+i_{1})})_{\fq}$$ is Cohen-Macaulay. That is,
$$
\text{non-CM}(R^{(p^{e_{2}},i_{2})})\subseteq \text{non-CM}\big(R^{(p^{e_{1}},i_{1})}\big)\backslash\{\fq\}\subseteq\text{non-CM}(R)\backslash\{\fp,\fq\}.
$$
Since non-CM$(R)$ is Zariski-closed, it has finitely many minimal members. So by repeating this procedure, we may assume that the codimension of the non-Cohen-Macaulay locus of some $R^{(p^{e_{n}},i_{n})}$ is strictly greater than that of $R$. Proceeding in this way, we may increase strictly the codimension of the non-Cohen-Macaulay locus of some  $F$-components, to max it out to be $\infty$, and thus to arrive eventually at some Cohen-Macaulay $F$-component $R^{(p^{e_{b}},i_{b})}$ of $R$ for some large $e_b$ and $i_b$. This is a contradiction.
\end{proof}

\begin{remark}
\begin{enumerate}[(i)]
\item
By taking the $\mathfrak{p}R_\mathfrak{p}$-adic completion of $R_\mathfrak{p}$ and $(R^{(p^e,k)})_{\mathfrak{p}}$ in the proof of Proposition \ref{PropositionGradedTrueNonGradedWrong}, one can conclude that there is a complete local $F$-finite ring admitting a generalized Cohen-Macaulay module $M$ as in Proposition \ref{PropositionGradedTrueNonGradedWrong} such that no direct summand $M_{e,i}$ of $F^e_*(M)$ is Cohen-Macaulay. 

\item
It is readily seen that the class of $F$-decomposing modules introduced in \cite[page 272]{SchoutensHochster} is a subclass of finitely generated modules $M$ for which $F^e_*(M)$ decomposes into $n_e$ number of non-zero direct summands with $\lim\limits_{e\rightarrow \infty}n_e=\infty$. The localized Veronese modules involved in the proof of Proposition \ref{PropositionGradedTrueNonGradedWrong} are instances of $F$-decomposing modules. 

\item
One may ask whether the non-graded version of Proposition \ref{PropositionAlternativeProof1} holds for generalized Cohen-Macaulay localizations of $R$ or $M$ by adding more conditions on $R$ or $M$. An affirmative answer to this question might be useful to the small Cohen-Macaulay conjecture in a positive direction.  One answer would be to have some $M$ with a similar decomposition such that $\lim\limits_{e \to \infty }\frac{n_e}{p^{\alpha e}}=\infty$, where $\alpha=\max\{\log_{p}[k(\fp):k(\fp)^{p}]:\fp\in \Spec R\}$ (that is a finite number) and $k(\fp)=R_\fp/\fp R_\fp$. However, this condition, on the one hand, is possibly stronger than what is really sufficient and on the other hand, it is apparently out of reach. For instance, in the case of Corollary \ref{CorollaryTrivialDeformation}, each $F_*^e(R[X])$ decomposes into at most $p^{2e}$ number of non-zero direct summands and some (non-Veronese) direct summand of $F_*^e(R[X])$ is Cohen-Macaulay. 
\end{enumerate}
\end{remark}

Let us give an alternative proof of Proposition \ref{PropositionAlternativeProof1}. To ensure that there is a homogeneous (full) system of parameters for $R$, we need to replace the $\mathbb{Z}$-graded assumption with being positively graded over an $F$-finite field $R_{[0]}$ in the next proposition.

\begin{proposition}
\label{UnExpectedResult}
Suppose that $R$ is a generalized Cohen-Macaulay  $\mathbb{N}_0$-graded local domain whose degree zero part is an $F$-finite field $K$ of characteristic $p>0$ with $[K:K^p]=p^\alpha$.  Then some direct summand $R^{(p^e,i)}$ of $F^e_*(R)$ is a maximal Cohen-Macaulay module.
\end{proposition}

\begin{proof}
We may and do assume that $R$ has positive dimension. Set $R_{K^\infty}:=R\otimes_{K}K^\infty$. Then $R_{K^\infty}$ is an $\mathbb{N}_0$-graded generalized Cohen-Macaulay ring with perfect residue field of positive characteristic. Notice that it is not necessary to assume that $R_{K^\infty}$ is a domain, because this property was used to say that almost all direct summands $R_{K^\infty}^{(p^e,i)}$ are non-zero modules of dimension $\dim R$. However, this latter desired property transfers from $R$ to $R\otimes_{K}K^\infty$ without the domain assumption. The same argument is valid when applying  Theorem \ref{Hochster} to $R_{K^\infty}$ to deduce that some $R_{K^\infty}^{(p^e,i)}$ is Cohen-Macaulay.
	
So $N:={R_{K^\infty}}^{(p^e,i)}$ is an $\mathbb{N}_0$-graded maximal Cohen-Macaulay $R_{K^\infty}$-module for some choices of $e$ and $i$ by Theorem \ref{Hochster}. There is a natural surjection of $\mathbb{N}_0$-graded $R_{K^\infty}$-modules
$$
\eta:M:=R^{(p^e,i)} \otimes_K K^\infty \twoheadrightarrow N:={R_{K^\infty}}^{(p^e,i)}~\mbox{defined by}~r_j\otimes k^{p^{-n}}\mapsto r_j\otimes k^{p^{e-n}}.
$$
Notice that each graded part $M_{[j]}=F^e_*(R_{[j]})\otimes_K K^\infty$ (with $j \equiv i \pmod{p^e}$) of $M$ has $K^\infty$-vector space dimension: $p^{\alpha e}\cdot \dim_K R_{[j]}$, while the $K^\infty$-vector space dimension of
$$
N_{[j]}=F^e_*(R_{[j]}\otimes_KK^\infty)
$$
agrees with $\dim_K R_{[j]}$. It follows that the Hilbert series of the kernel $L$ of the $R_{K^\infty}$-epimorphism $\eta$ satisfies
\begin{equation}
\label{ComparisionOfHilbertSeries}
\text{H}_{L}(t)=(p^{\alpha e}-1)\text{H}_{N}(t).
\end{equation}
Pick a homogeneous system of parameters $\mathbf{x}$ for $R$ and let $d_i$ denote the degree of $x_i$. Let $A$ be a graded Noether normalization for $R$ whose regular system of parameters is $\mathbf{x}$, which gives the normalization $A_{K^\infty} \rightarrow R_{K^\infty}$ with the same regular system of parameters $\mathbf{x}$. Applying the functor $-\otimes_{A_{K^\infty}}\big(A_{K^\infty}/\mathbf{x}A_{K^\infty}\big)$ to the exact sequence $0 \rightarrow L\hookrightarrow M\twoheadrightarrow N\rightarrow 0$, we get the exact sequence: 
\begin{equation}
\label{VeryUsefulExactSequence}
0 \to L/\mathbf{x}L \to M/\mathbf{x}M \rightarrow N/\mathbf{x}N \to 0,
\end{equation}
due to the Cohen-Macaulay property of $N$ and regularity of $A_{K^\infty}$. For any $j$ with $j \equiv i \pmod{p^e}$,
\begin{equation}
\label{FirstIdentity}
\dim_{K^\infty} \big(M/\mathbf{x}M\big)_{[j]}=\dim_{K^\infty} F^e_*\big(R_{[j]}/(\sum \limits_{i=1}^{\dim R} x_i^{p^e}R_{[j-d_ip^e]})\big)\otimes_K K^\infty=p^{\alpha e} \dim_{K}\ R_{[j]}/(\sum \limits_{i=1}^{\dim R}x_i^{p^e}R_{[j-d_ip^e]}),
\end{equation}
and likewise
\begin{equation}
\label{SecondIdentity}
\dim_{K^\infty}  \big(N/\mathbf{x}N\big)_{[j]}=\dim_{K^\infty} F^e_*\big(R_{[j]}/(\sum \limits_{i=1}^{\dim R}x_i^{p^e}R_{[j-d_ip^e]}) \otimes_K K^\infty\big)=\dim_{K} R_{[j]}/(\sum \limits_{i=1}^{\dim R}x_i^{p^e}R_{[j-d_ip^e]})
\end{equation} 
for any $j$. It follows that $\dim_{K^\infty} (M/\mathbf{x}M)=p^{\alpha e} \dim_{K^\infty} (N/\mathbf{x}N)$.  Moreover, in view of $(\ref{VeryUsefulExactSequence})$, we have
\begin{align*}
\mu_{A_{K^\infty}}(L)=\dim_{K^\infty}(L/\mathbf{x}L)=\dim_{K^\infty}(M/\mathbf{x}M)-\dim_{K^\infty}(N/\mathbf{x}N)&=(p^{\alpha e}-1)\dim_{K^\infty}(N/\mathbf{x}N)\\&=(p^{\alpha e}-1)\mu_{A_{K^\infty}}(N).
\end{align*}  
In view of $(\ref{FirstIdentity})$ and $(\ref{SecondIdentity})$, there is a degree-preserving one-to-one correspondence between the minimal generators of $L$ and those of $\bigoplus_{i=1}^{p^{\alpha e}-1}N$.  In conjunction with the fact that $N$ is a graded free $A_{K^\infty}$-module, this implies that there is a well-defined homogeneous $A_{K^\infty}$-epimorphism $\bigoplus_{i=1}^{p^{\alpha e}-1}N\twoheadrightarrow  L$ whose kernel has to be trivial. If otherwise, the identity $(\ref{ComparisionOfHilbertSeries})$ is violated.
\end{proof}


\begin{thebibliography}{99}
\bibitem[Ab01]{AberbachExtension}
I. M. Aberbach, \emph{Extension of weakly and strongly F-regular rings by flat maps}, J. Algebra, \textbf{241} (2001), 799--807.


\bibitem[Ao83]{AoyamaSomeBasic}
Y. Aoyama, \emph{Some basic results on canonical modules}, J. Math. Kyoto Univ., \textbf{23} (1983), 85--94.


\bibitem[AoGo85]{AoyamaGotOnTheEndomorphism}
Y. Aoyama and S. Goto, \emph{On the endomorphism ring of the canonical module}, J. Math. Kyoto Univ., \textbf{25} (1985), 21--30.


\bibitem[ArDeHa14]{CoxRing}
I. Arzhantsev, U. Derenthal, and J. Hausen, \emph{Cox rings},
Cambridge University Press, Cambridge, 2014.


\bibitem[BrSh13]{BrodmannSharpLocal}
M. P. Brodmann and  R. Y. Sharp, \emph{Local Cohomology: An Algebraic Introduction with Geometric Applications, second edition}, Cambridge University Press, (2013).
   

\bibitem[BrHe98]{BrunsHerzogCohenMacaulay}
W. Bruns and J. Herzog, \emph{Cohen-Macaulay rings}, Cambridge, (1998).


\bibitem[De79]{DemazureAnneaux}
M. Demazure, \emph{Anneaux gradues normaux}, in Seminaire Demazure-Giraud-Teissier, Singularites des surfaces, Ecole Polytechnique, 1979.


\bibitem[Fo73]{FossumTheDivisor}
R. M. Fossum, \emph{The divisor class group of a Krull domain}, Ergebnisse der Mathematik und ihrer Grenzgebiete \textbf{74}, Springer-Verlag, New York-Heidelberg, 1973.


\bibitem[G\"oWe10]{GortzWedhornAlgebraic} U. G\"ortz and T. Wedhorn, \emph{Algebraic geometry {I}. {S}chemes---with examples and exercises}, Advanced Lectures in Mathematics, Wiesbaden: Vieweg+Teubner, (2010).


\bibitem[GoWa78]{GotoWatanabeOnGraded}
S. Goto and K. Watanabe, \emph{On graded rings I}, J. Math. Soc. Japan \textbf{30} (1978), 179--213.


\bibitem[Gr91]{GriffithSomeResults} P. Griffith, \emph{Some results in local rings on ramification in low codimension}, J. Algebra, \textbf{137} (1991), 473--490.


\bibitem[Gr61]{GrothendieckEGAIII} 
A. Grothendieck,  \emph{\'El\'ements de g\'eom\'etrie alg\'ebrique III. \'Etude cohomologique des faisceaux coh\'erents, Premi\`ere partie}, Publications Math. I.H.E.S. \textbf{21} (1961).


\bibitem[Gr65]{GrothendieckEGAIV} 
A. Grothendieck,  \emph{\'El\'ements de g\'eom\'etrie alg\'ebrique IV. 	
\'Etude locale des sch\'emas et des morphismes de sch\'emas, Seconde partie}, Publications Math. I.H.E.S. \textbf{21} (1965).


\bibitem[Gr67]{GrothendieckEGAIV2} 
A. Grothendieck,  \emph{\'El\'ements de g\'eom\'etrie alg\'ebrique IV. 	
\'Etude locale des sch\'emas et des morphismes de sch\'emas, Quatri\`eme partie}, Publications Math. I.H.E.S. \textbf{32} (1967).


\bibitem[Ha83]{HartshorneAlgebraicGeometry}
R. Hartshorne, \emph{Algebraic geometry},  Graduate Texts in Mathematics, 52 Springer (1983).


\bibitem[HaOg74]{HartshorneOgusOnTheFactoriality} R. Hartshorne and A. Ogus, \emph{On the factoriality of local rings of small embedding codimension}, \textbf{1} (1974), 415--437.


\bibitem[HeTr93]{HerrmannTrungExamples} M. Herrmann  and N. V. Trung, \emph{Examples of Buchsbaum quasi-Gorenstein rings}, Proc. Am. Math. Soc., \textbf{117} (1993),  619--625.


\bibitem[Ho04]{HochsterCurrent}
M. Hochster, \emph{Current state of the homological conjectures}, Tech. Report, University of Utah, \url{http://www.math.utah.edu/vigre/minicourses/algebra/hochster.pdf}, (2004).


\bibitem[ImSc13]{ImtiazSchenzel}
M. Imtiaz and P. Schenzel, \emph{On the non-Cohen-Macaulayness of certain factorial closures}, Commun. Algebra, \textbf{41} (2013), 3397--3413.


\bibitem[IyMaScWa21]{IyengarMaEtAl} S. B. Iyengar, L. Ma, K. Schwede and M. E. Walker, \textit{Maximal Cohen-Macaulay complexes and their uses: A partial survey}, In: Commutative algebra. Expository papers dedicated to David Eisenbud on the occasion of his 75th birthday, Cham: Springer, pp. 475--500,  (2021).


\bibitem[Ka99]{KatzOnTheExistence}
D. Katz, \emph{On the existence of maximal {Cohen}-{Macaulay} modules over {{\(p\)}}-th root extensions}, Proc. Am. Math. Soc., \textbf{9} (1999), 2601--2609.


\bibitem[Ke80]{KempfSomeElementaryProofs}
G. R. Kempf, \textit{Some elementary proofs of basic theorems in the cohomology of quasi-coherent sheaves}, Rocky Mountain J. Math., \textbf{10} (1980), 637--645.


\bibitem[La04]{LazarsfeldPositivity}
R. Lazarsfeld, \emph{Positivity in algebraic geometry. I. Classical setting: line bundles and linear series}, Ergebnisse der Mathematik und ihrer Grenzgebiete. 3. Folge, \textbf{48} (2004).


\bibitem[Ma19]{MaMaximalCM}
L. Ma, \emph{Maximal Cohen-Macaulay modules over certain Segre products}, Commun. Algebra \textbf{47} (2019), 2488-2493.


\bibitem[MaSc11]{MarceloSchenzelNonCMUFDs}
A. Marcelo and P. Schenzel, \emph{Non-Cohen–Macaulay unique factorization domains in small dimensions}, J. Symb. Comput., \textbf{46} (2011), 609-621.


\bibitem[Ma89]{Matsumura}
H. Matsumura, \emph{Commutative ring theory},  Cambridge University Press, Cambridge, (1989).


\bibitem[NhQu14]{NhanQuyAttached} L. T. Nhan and  P. H. Quy, \textit{Attached primes of local cohomology modules 	under localization and completion}, J. Algebra, \textbf{420} (2014), 475--485.


\bibitem[Ro80]{RobertsCohenMacaulay}
P. Roberts, \emph{Cohen-Macaulay complexes and an analytic proof of the new intersection conjecture}, J. Algebra, \textbf{66} (1980), 220--225.


\bibitem[Ro84]{RobbianoFactorial} L. Robbiano, \emph{Factorial and almost factorial schemes in weighted projective spaces}, Lecture Notes in Math., vol. 1092, Springer (1984), 62--84.


\bibitem[Ro09]{Rotman}
J. Rotman, \emph{An introduction to homological algebra. 2nd ed.}, Berlin: Springer, 2009.


\bibitem[ScSi18]{SchenzelSimonCompletion}
P. Schenzel and A.-M. Simon, \emph{Completion, \v{C}ech and local homology and cohomology. Interactions between them}, Cham: Springer (2018).


\bibitem[Sc17]{SchoutensHochster}
H. Schoutens, \emph{Hochster's small MCM conjecture for three-dimensional weakly F-split rings}, Commun. Algebra, \textbf{45} (2017), 262--274.


\bibitem[Sc19]{SchoutensADifferential}
H. Schoutens, \emph{A differential-algebraic criterion for obtaining a small maximal Cohen-Macaulay module}, Proc. Amer. Math. Soc., \textbf{148}, No. 10 (2020), 4165--4177.


\bibitem[Sh11]{ShimomotoFCoherent}
K. Shimomoto, \emph{F-coherent rings with applications to tight closure theory}, J. Algebra, \textbf{338} (2011), 24--34.


\bibitem[ShTaTa20]{ShimomotoTaniguchiTavanfar}
K. Shimomoto, N.  Taniguchi, and E. Tavanfar, \emph{A study of quasi-Gorenstein rings. II: Deformation of quasi-Gorenstein property}, J. Algebra, \textbf{562} (2020), 368-389.


\bibitem[ShTa21]{ShimomotoTavanfarOnLocalRings}
K. Shimomoto and E. Tavanfar, \emph{On local rings without small Cohen-Macaulay algebra in mixed characteristic},  \url{https://arxiv.org/abs/2109.12700}.


\bibitem[Sr21-1]{SridharExistence}
P. Sridhar, \emph{Existence of birational small Cohen-Macaulay modules over biquadratic extensions in mixed characteristic}, J. Algebra, \textbf{582}  (2021), 	100-116.


\bibitem[Sr21-2]{SridharOnThe}
P. Sridhar, \emph{On the existence of birational maximal Cohen-Macaulay modules over biradical extensions in mixed characteristic}, J. Pure Appl. Algebra, \textbf{225}  (2021).


\bibitem[StVo78]{StuckradVogelBuchsbaum}
J. St\"{u}ckrad and W. Vogel, \emph{Buchsbaum rings and applications. An interaction between algebra, geometry and topology},  Berlin etc.: Springer-Verlag (1986).


\bibitem[Ta17]{TavanfarReduction}
E. Tavanfar, \emph{Reduction of the small Cohen-Macaulay conjecture to excellent unique factorization domains}, Arch. Math. \textbf{109} (2017), 429--439. 



\bibitem[Ta]{TavanfarSmall}
E. Tavanfar, \emph{Small Cohen-Macaulay conjecture and blow-up algebras of excellent regular local rings}, in progress.


\bibitem[Ta21]{TavanfarTest}
E. Tavanfar, \emph{Test modules, weakly regular homomorphisms and complete intersection dimension}, \url{arXiv:1911.11290v3}, to appear in J. Commut. Algebra.



\bibitem[TaTo18]{TavanfarTousiAStudy}
E. Tavanfar and M. Tousi, \emph{A study of quasi-Gorenstein rings}, J. Pure App. Algebra, \textbf{222} (2018), 3745--3756.


\bibitem[Wa81]{WatanabeSomeRemarks}
K. Watanabe, \emph{Some remarks concerning {D}emazure's construction of normal graded rings}, Nagoya Math. J. \textbf{83} (1981), 203--211.
\end{thebibliography}
\end{document}